\documentclass[11pt,reqno]{amsproc}
\usepackage[margin=1in]{geometry}
\usepackage{amsmath, amsthm, amssymb}
\usepackage[colorlinks=true, pdfstartview=FitV, linkcolor=blue,citecolor=blue, urlcolor=blue,bookmarksdepth=2]{hyperref}
\usepackage[abbrev,lite,nobysame]{amsrefs}
\usepackage[usenames,dvipsnames]{color}
\usepackage{mathtools}

\title[Convergence of non-autonomous attractors for subquintic  wave equation]{Convergence of non-autonomous attractors for subquintic weakly damped wave equation}
\author{Jakub Bana\'{s}kiewicz, Piotr Kalita}
\address{Faculty of Mathematics and Computer Science, Jagiellonian University, ul. Łojasiewicza 6, 30-348, Kraków}
\email{Jakub.Banaskiewicz@im.uj.edu.pl, piotr.kalita@ii.uj.edu.pl}
\thanks{This work was supported by National Science Center (NCN) of Poland under projects No. DEC-2017/25/B/ST1/00302 and UMO-2016/22/A/ST1/00077.}

\date{\today}

\newcommand{\norm}[1]{\|{#1}\|}
\newcommand{\set}[1]{\{{#1}\}}

\newcommand{\ut}[0]{u_{t}}
\newcommand{\dt}[0]{\frac{d}{dt}}
\newcommand{\R}[0]{\mathbb{R}}
\newcommand{\CC}[0]{C(\mathbb{R};C^1(\mathbb{R}))}
\newcommand{\dist}[0]{\text{dist}}
\newcommand{\partu}[1]
{\frac{\partial {#1}}
{\partial u}
}

\newtheorem{theorem}{Theorem}[section]
\newtheorem{corollary}{Corollary}[theorem]
\newtheorem{lemma}[theorem]{Lemma}
\newcommand{\fraclap}[2]{(-\Delta)^{\frac{{#1} }{{#2} }}}
\newtheorem{proposition}{Proposition}[section]
\theoremstyle{definition}
\newtheorem{definition}[proposition]{Definition}
\newtheorem{remark}[proposition]{Remark}

\numberwithin{equation}{section}

\begin{document}
\begin{abstract}
We study the non-autonomous weakly damped wave equation with subquintic growth condition on the nonlinearity. Our main focus is the class of Shatah--Struwe solutions, which satisfy the Strichartz estimates and are coincide with the class of solutions obtained by the Galerkin method. For this class we show the existence and smoothness of pullback, uniform, and cocycle attractors  and the relations between them. We also prove that these non-autonomous attractors converge upper-semicontinuously to the global  attractor for the limit autonomous problem if the time-dependent nonlinearity tends to time independent function in an appropriate way.

\end{abstract}

\maketitle
\section{Introduction.} In this paper we are interested in the existence, regularity and upper-semicontinuous convergence of pullback, uniform and cocycle attractors of the problems governed by the following weakly damped wave equations
\begin{equation}\label{eq:nonauto}
u_{tt} + u_t - \Delta u = f_\varepsilon(t,u).
\end{equation}
We prove that these attractors converge as $\varepsilon\to 0$ to the global attractor of the problem governed by the limit autonomous equation
\begin{equation}\label{eq:auto}
u_{tt} + u_t - \Delta u = f_0(u),
\end{equation}
where $f_\varepsilon \to f_0$ in appropriate sense.
The unknowns are the functions $u:[t_0,\infty) \times \Omega \to \R$, where $\Omega$ is an open and bounded domain with smooth boundary.

The theory of global attractors for the wave equation with damping term $u_t$ has been developed by Babin and Vishik \cite{Babin-Vishik-1983}, Ghidaglia and Temam \cite{Ghidaglia_Temam}, Hale \cite{Hale_1985}, Haraux \cites{Haraux1, Haraux2}, Pata and Zelik \cite{PataZelik}. Overview of the theory can be found, among others, in the monographs of Babin and Vishik \cite{Babin_Vishik}, Haraux \cite{Har_book}, Chueshov and Lasiecka \cite{Chueshov_Lasiecka}. We also mention the classical monographs of Henry \cite{Henry-1981}, Hale \cite{Hale-1988}, Robinson \cite{Robinson}, Temam \cite{Temam}, and D\l{}otko and Cholewa \cite{Dlotko_Cholewa} on infinite dimensional autonomous dynamical systems. Various types of non-autonomous attractors and their properties has been studied, among others, by Chepyzhov and Vishik \cite{Chepyzhov_Vishik}, Cheban \cite{Cheban}, Kloeden and Rasmussen \cite{Kloeden}, Carvalho, Langa, and Robinson \cite{Carvalho-Langa-Robinson-2012}, Chepyzhov \cite{Chepyzhov-2013}, and Bortolan, Carvalho and Langa \cite{BoCaL}.

The existence of the global attractor for \eqref{eq:auto} with the cubic growth condition 
\begin{equation}\label{eq:growth_3}
|f_0(s)| \leq C(1+|s|^3),
\end{equation}
has been obtained by Arrieta, Carhalho, and Hale in \cite{Arrieta-Carvalho-Hale-1992}. This growth exponent had long been considered as critical. In 2016 Kalantarov, Savostianov and Zelik \cite{Savostianov} used the findings on the Strichartz estimates for the wave equation on bounded domains \cites{Burq, Blair} to obtain the global attractor existence for the so called Shatah--Struwe solutions of quintic weakly damped wave equation, i.e. where  the exponent $3$ in  \eqref{eq:growth_3} is replaced by $5$. These findings led to the rapid development of the theory for weakly damped wave equation with supercubic growth.
In particular, global attractors for Shatah--Struwe solutions for supercubic case  with forcing in $H^{-1}$ have been studied by Liu, Meng, and Sun \cite{LMS}, and the exponential attractors were investigated by Meng and Liu in \cite{Meng_Liu}. We also mention the work \cite{Carvalho-Chol-Dlot-2009} of Carvalho, Cholewa, and D\l{}otko who obtained an existence of the weak global attractor for a concept of solutions for supercubic but subquintic case. 
Finally, the results on attractors for autonomous problems with supercubic nonlinearities have been generalized to the case of damping given by the fractional Laplacian in the subquintic case in \cite{Savo2} and in the quintic case in \cite{Savo1}. 

For a non-autonomous dynamical system there exist several important concepts of attractors: the pullback attractor, a time-dependent family of compact sets attracting ``from the past" \cites{Carvalho-Langa-Robinson-2012, Kloeden}, the uniform attractor, the minimal compact set attracting forwards in time uniformly with respect to the driven system of non-autonomous terms \cite{Chepyzhov_Vishik}, and the cocycle attractor which, in a sense unifies and extends the two above concepts \cites{Kloeden, Langa}. An overview of these notions can be found in the review article \cite{balibrea}.  Recent intensive research on the characterization of pullback attractors and continuity properties for PDEs \cites{Carvalho-Langa-Robinson-2012, Kloeden, Langa} has led to the results on the link between the notions of uniform, pullback, and cocycle attractors, namely an internal characterization of a uniform attractor as the union of the pullback attractors related to all their associated symbols (see \cite{Langa}, and Theorem \ref{NoAtr-Rel} below), and thus allowing to define the notion of lifted invariance (see \cite{Langa}, and Definition \ref{lifted} and Theorem \ref{lifted-inv} below) on uniform attractors. 

There are several recent results on nonautonomous version of weakly damped wave equation with quintic, or at least supercubic, growth condition which use the concept of  Shatah--Struwe solutions. Savostianov and Zelik in the article \cite{Savostianov_measure} obtain the existence of the uniform attractor for the problem governed by 
$$
u_{tt} + u_t + (1 - \Delta) u + f(u) = \mu(t),
$$
on the three dimensional torus, where $\mu(t)$ can be a measure. Mei, Xiong, and Sun \cite{mcs} obtain the existence of the pullback attractor for the problem governed by the equation 
\begin{equation}\label{auto}
u_{tt} + u_t - \Delta u + f(u) = g(t),
\end{equation}
for the subquintic case on the space domain given by whole $\R^3$ in the so called locally uniform spaces. Mei and Sun \cite{Mei_Sun} obtain the existence of the uniform attractor for non translation compact forcing for the problem governed by
\eqref{auto} with subquintic $f$. Finally, Chang, Li, Sun, and Zelik \cite{clsz} consider the problem of the form 
$$
u_{tt} + \gamma(t)u_t-\Delta u + f(u) = g,
$$
and show the existence of several types of nonautonomous attractors with quintic nonlinearity  for the case where the damping may change sign. None of these results consider the nonlinearity of the form $f(t,u)$ and none of these results fully explore the structure of non-autonomous attractors and relation between pullback, uniform, and cocycle attractors characterized in \cite{Langa}. The present paper aims to fill this gap.

In this article we generalize the results of \cite{Savostianov} to the problem governed by the weakly damped nonautonomous wave equation  \eqref{eq:nonauto} with the semilinear term $f_\varepsilon(t,u)$ which is a perturbation of the autonomous nonlinearity $f_0(u)$, cf. assumptions \hyperref[as:H2]{(H2)} and \hyperref[as:H3]{(H3)} in Section \ref{sec:3}. We stress that we deal only with the case of the  subquintic growth
$$
\left|\frac{\partial f_\varepsilon}{\partial u}(t,u)\right| \leq C(1+|u|^{4-\kappa}), \
$$
for which we prove the results on the existence and asymptotic smoothness of Shatah--Struwe solutions, derive the asymptotic smoothing estimates and obtain the result on the upper-semicontinuous  convergence of attractors. Thus we extend and complete the previous results in \cite{Savostianov} where only the autonomous case was considered and in \cites{Mei_Sun, mcs} where the nonlinearity was only in the autonomous term. We stress some key difficulties and achievements of our work. We follow the methodology of \cite[Proposition 3.1 and Proposition 4.1]{Savostianov} to derive the Strichartz estimate for the nonlinear problem from the one for the linear problem (where we use the continuation principle that can be found for example in \cite[Proposition 1.21]{Tao}) but in the proof we need the extra property that the constant $C_h$ in the linear Strichartz estimate
$$
\|u\|_{L^4(0,h;L^{12})} \leq C_h(\|(u_0,u_1)\|_{\mathcal{E}_0} + \|G\|_{L^1(0,T;L^{2})}),
$$
is a nondecreasing function of $h$. We establish this fact with the use of the Christ--Kiselev lemma \cite[Lemma 3.1]{Sogge}. 
As the part of the definition of the weak solution that we work with, we choose that it should be the limit of the Galerkin approximations. In \cite[Section 3]{Savostianov} the authors decide to work with Shatah--Struwe solutions (i.e. the solutions posessing the extra $L^4(0,T;L^{12}(\Omega))$ regularity), and they prove that each such solution must be the limit of the Galerkin approximations, cf. \cite[Corollary 3.6]{Savostianov}. We establish that in the subcritical case the two notions are in fact equivalent, cf. our Lemma \ref{lem:savostianov}. In \cite[Corollary 4.3]{Savostianov} the authors derive only $\mathcal{E}_\delta$ estimates saying in Remark 4.6 about possibility of further bootstrapping arguments. We derive in Section \ref{sec:6} the relevant asymptotic smoothing estimates and thus provide the result on the attractor smoothness in $\mathcal{E}_1$. The main result of the paper about non-autonomous attractors, cf. Theorem \ref{thm:main}, uses the findings of \cites{Kloeden, Langa} and establishes the existence, smoothness, and relation between uniform, cocycle (and thus also pullback) attractors. Finally, another novelty of the present paper is the  upper-semicontinuity result of Section \ref{sec:7} which also concerns these three classes of non-autonomous attractors. 

The possible extension of our results can involve dealing with a non-autonomous nonlinearity with critical quintic growth condition. This case is more delicate because control of the energy norm of the initial data does not give the control over norm $L^4(0,T;L^{12}(\Omega))$. To overcome this problem Kalantarov, Savostianov, and Zelik in \cite{Savostianov} used the technique of trajectory attractors. Another interesting question is the possibility of extending the results of
\cite{Mirelson-Kalita-Langa} about the convergence of non-autonomous attractors for equations
$$\varepsilon u_{tt}+u_{t}-\Delta u =f_\varepsilon(t,u)$$
to the attractor for the semilinear heat equation as $\varepsilon\to 0$,
in the case of subquintic or quintic growth condition on $f$. The main difficulty is to obtain the Strichartz estimates which are uniform with respect to $\varepsilon$.
Finally we mention the possible further line of research involving the lower semicontinuous convergence of attractor and the stability of the attractor structure under perturbation.  

Structure of the article is as follows. After some preliminary facts reminded in Section \ref{sec:2}, the formulation of the problem, assumptions of its data, and some auxiliary results regarding the translation compactness on the non-autonomous term are presented in Section \ref{sec:3}. Next Section \ref{sec:4} is devoted to the Galerkin solutions and their dissipativity and the following Section \ref{sec:5} contains the results on the Strichartz estimates, Shatah--Struwe solutions, and their equivalence with the Galerkin solutions. Result on the existence and asymptotic smoothness of non-autonomous attractors, Theorem \ref{thm:main} is contained in Section \ref{sec:6}, while in Section \ref{sec:7}  
we prove their upper-semicontinuous convergence to the global attractor of the  autonomous problem. Some auxiliary results needed in the paper are included in the final Section \ref{sec:8}.

\section{Preliminaries.}\label{sec:2}
Let $\Omega\subset\mathbb{R}^3$ be a bounded and open set with sufficiently
smooth boundary. We will use the notation $L^2$ for $L^2(\Omega)$ and in
general for notation brevity we will skip writing dependence on $\Omega$ in
spaces of functions defined on this set.
By $(\cdot,\cdot),\norm{.}$ we will denote respectively the scalar product and
 the norm in $L^2$.
 We will also use the notation $\mathcal{E}_0 = H^1_0\times L^2$ for the energy space. Its norm is defined by $\|(u,v)\|^2_{\mathcal{E}_0} = \|\nabla u\|^2 + \|v\|^2$. In the article by $C$ we denote a generic positive constant which can vary from line to line. We recall some useful information concerning the spectral fractional Laplacian \cite{Antil}. Denote by $\{e_i\}_{i=1}^\infty$ the eigenfunctions (normed to $1$ in $L^2(\Omega)$) of the operator $-\Delta$ with the Dirichlet boundary conditions, such that the corresponding eigenvalues are given by
$$
0< \lambda_1 \leq \lambda_2 \leq \ldots \leq \lambda_n \leq \ldots.
$$
For $u\in L^2$ its $k$-th Fourier coefficient is defined as $\widehat{u}_k = (u,e_k)$. Let $s\geq 0$. The spectral fractional laplacian is defined by the formula
$$
(-\Delta)^\frac{s}{2}u = \sum_{k=1}^\infty \lambda_k^\frac{s}{2} \widehat{u}_k.
$$
The space $\mathbb{H}^s$ is defined as
$$
\mathbb{H}^s = \left\{  u\in L^2\,:\ \sum_{k=1}^\infty \lambda_k^s \widehat{u}_k^2 < \infty\right\}.
$$
The corresponding norm is given by
$$\|u\|_{\mathbb{H}^s} = \|(-\Delta)^{s/2}u\| = \sqrt{\sum_{k=1}^\infty \lambda_k^s \widehat{u}_k^2}.$$
The space $\mathbb{H}^s$ is a subspace of the fractional Sobolev space $H^s$. In particular
$$
\mathbb{H}^s = \begin{cases}
H^s = H^s_0\ \ \textrm{for}\ \ s\in (0,1/2),\\
H^s_0 \ \ \textrm{for}\ \ s\in (1/2,1].
\end{cases}
$$
We also remind that the standard fractional Sobolev norm satisfies $\|u\|_{H^s}\leq C\|u\|_{\mathbb{H}^s}$ for $u\in \mathbb{H}^s$, cf. \cite[Proposition 2.1]{Antil}.
For $s\in [0,1]$ we will use the notation $\mathcal{E}_s = \mathbb{H}^{s+1} \times \mathbb{H}^s$. This space is equipped with the norm $\|(u,v)\|^2_{\mathcal{E}_s} = \|u\|^2_{\mathbb{H}^{s+1}} + \|v\|^2_{\mathbb{H}^s}$. 
 
\section{Problem definition and assumptions.}\label{sec:3}
We consider the following family of  problems parameterized by $\varepsilon > 0$
\begin{equation}\label{eq:prblm}
\begin{cases}
 u_{tt}+u_t-\Delta u = f_\varepsilon(t,u)\; \text{for}\; (x,t)\in \Omega\times(0,\infty),\\
u(t,x) = 0\; \text{for}\; x\in \partial \Omega,
\\u(0,x) = u_0(x),
\\u_t(0,x) = u_1(x).
\end{cases}
\end{equation}
The initial data has the regularity $(u_0,u_1) \in \mathcal{E}_0$.
Throughout the article we always assume that the nonautonomous and
nonlinear term $f_\varepsilon(t,u)$, treated as the mapping
which assigns to the time $t\in \R$ the function of the variable $u$,
belongs to the space $C(\mathbb{R};C^1(\mathbb{R}))$.
This space is equipped with the metric
$$d_{C(\mathbb{R};C^1(\mathbb{R}))}(g_1,g_2)
= \sum_{i=1}^\infty\frac{1}{2^i}
\frac
{\sup_{t\in[-i,i]}
d_{C^1(\mathbb{R})}(g_1(t,.),g_2(t,.))
}
{1+\sup_{t\in[-i,i]}
d_{C^1(\mathbb{R})}(g_1(t,.),g_2(t,.))
}
\;\text{for $g_1,g_2\in C(\mathbb{R};C^1(\mathbb{R}))$}
,$$
where the metric in $C^1(\mathbb{R})$ is defined as follows
$$d_{C^1(\mathbb{R})}(g_1,g_2)
=\sum_{i=1}^\infty\frac{1}{2^i}
\frac
{
\norm{g_1(u)-g_2(u)}_{C^1([-i,i])}
}
{1+
\norm{g_1(u)-g_2(u)}_{C^1([-i,i])}
}
\;\text{for $g_1,g_2\in C^1(\mathbb{R})$},
$$
and $\|g\|_{C^1(A)} = \max_{r\in A}|g(r)| + \max_{r\in A}|g'(r)|$ for a compact set $A\subset \R$.
\begin{remark}
If $g_n\to g$ in sense of $\CC$ then $ g_n\to g$ and $\partu{ g_n}\to \partu{g}$ uniformly on every bounded subset of $\mathbb{R}$.
\end{remark}
We make the following assumptions on functions $f_\varepsilon:\R\times \R\to \R$ and $f_0:\R\to \R$
\begin{itemize}
    \item[(H1)]\label{as:H1}
    For every $\varepsilon \in (0,1]$ the function $f_{\varepsilon}\in C(\mathbb{R};C^1(\mathbb{R}))$, and  $f_0\in C^1(\mathbb{R})$.

    \item[(H2)]\label{as:H2}
    For every $u\in \R$
$$
\lim_{\varepsilon\to 0} \sup_{t\in\mathbb{R}}|f_\varepsilon(t,u)-f_0(u)| = 0.
$$
\item[(H3)]\label{as:H3} There holds
$$
\sup_{\varepsilon\in[0,1]}\sup_{t\in \R}\sup_{u\in \R} |f_\varepsilon(t,u)-f_0(u)| < \infty.
$$
\item[(H4)]\label{as:H4} There holds
$$
\limsup_{|u|<\infty} \frac{f_0(u)}{u} < \lambda_1,
$$
where $\lambda_1$ is the first eigenvalue of $-\Delta$ operator with the Dirichlet boundary conditions.
\item[(H5)]\label{as:H5}
There exists $\kappa > 0$ and $C>0$ such that

$$
\sup_{\varepsilon\in[0,1]}\sup_{t\in \R}\left|\frac{\partial f_\varepsilon}{\partial u}(t,u)\right| \leq C(1+|u|^{4-\kappa})\ \ \textrm{for every}\ \ u\in \R.
$$

\item[(H6)]\label{as:H6}  For any fixed $u\in\mathbb{R}$ the map $f_\varepsilon(t,u)$ is uniformly continuous with respect to $t$. Moreover for every $R>0$ the map $\R\times [-R,R]\ni (t,u)\mapsto \frac{\partial f_\varepsilon}{\partial u}(t,u)$ is uniformly continuous.

\end{itemize}
\begin{remark}
The example of family of functions satisfying conditions \hyperref[as:H1]{(H1)}-\hyperref[as:H6]{(H6)} is $f_\varepsilon(t,u)=-u|u|^{4-\kappa}+g(u)+\varepsilon \sin(t)\sin(u^3) $ where the growth of $g(u)$ is essentially lower than $5-\kappa$.
\end{remark}

\begin{proposition}\label{prop23}
	Assuming \hyperref[as:H1]{(H1)}, \hyperref[as:H3]{(H3)}, \hyperref[as:H5]{(H5)}, and \hyperref[as:H6]{(H6)}, for every $R > 0$ the mapping
	$$
	\R\times [-R,R]\ni (t,u) \mapsto f_{\varepsilon}(t,u)	$$
	is uniformly continuous.
\end{proposition}
\begin{proof}
	Let $u_1, u_2 \in [-R,R]$ and $t_1, t_2 \in \R$. There holds
\begin{align*}
&	|f_{\varepsilon}(t_1,u_1)-f_{\varepsilon}(t_2,u_2)| \leq 	|f_{\varepsilon}(t_1,u_1)-f_{\varepsilon}(t_1,u_2)| + 	|f_{\varepsilon}(t_1,u_2)-f_{\varepsilon}(t_2,u_2)| \\
& \qquad \leq C(R)|u_1-u_2| + \sup_{|u|\leq R}|f_{\varepsilon}(t_1,u)-f_{\varepsilon}(t_2,u)|.
	\end{align*}
	It suffices to prove that for every $\eta > 0 $ we can find $\delta > 0 $ such that if only $|t_1-t_2|\leq \delta$ then $\sup_{|u|\leq R}|f_{\varepsilon}(t_1,u)-f_{\varepsilon}(t_2,u)| \leq \eta$.
	Assume for contradiction that there exists $\eta_0 > 0$ such that for every $n\in \mathbb{N}$ we can find $t^n_1,t^n_2 \in \R$ with $|t_1-t_2|\leq \frac{1}{n}$ and
	$$
	\sup_{|u|\leq R}|f_{\varepsilon}(t^n_1,u)-f_{\varepsilon}(t^n_2,u)| >  \eta_0.
	$$
	For every $n$ there exists $u^n$ with $|u^n|\leq R$ such that
	$$
	|f_{\varepsilon}(t^n_1,u^n)-f_{\varepsilon}(t^n_2,u^n)| >  \eta_0.
	$$
	For a subsequence $u^n\to u^0$ with $|u^0|\leq R$, hence
	\begin{align*}
	&\eta _0 < 	|f_{\varepsilon}(t^n_1,u^n)-f_{\varepsilon}(t^n_1,u^0)| + 	|f_{\varepsilon}(t^n_1,u^0)-f_{\varepsilon}(t^n_2,u^0)| + 	|f_{\varepsilon}(t^n_2,u^0)-f_{\varepsilon}(t^n_2,u^n)|\\
	& \ \ \  \leq 2C(R) |u^n-u^0| + |f_{\varepsilon}(t^n_1,u^0)-f_{\varepsilon}(t^n_2,u^0)|.
	\end{align*}
	By taking $n$ large enough we deduce that
	$$
	\frac{\eta_0}{2} < |f_{\varepsilon}(t^n_1,u^0)-f_{\varepsilon}(t^n_2,u^0)|,
	$$
	a contradiction with uniform continuity of $f_{\varepsilon}(\cdot,u_0)$ assumed in \hyperref[as:H6]{(H6)}.
\end{proof}

We define hull of $f$ as a set $\mathcal{H}(f):=\overline{\{f(t+\cdot,\cdot)\in\CC \}_{t\in \R}}$, where the closure is understood in the metric $d_{C(\R;C^1(\R))}$. We also define set
$$\mathcal{H}_{[0,1]}:=\bigcup_{\varepsilon \in  (0,1]}\mathcal{H}(f_\varepsilon) \cup \{f_0\}$$
We say that a function $f$ is translation compact if its hull $\mathcal{H}(f)$ is a compact set. The following characterization of translation compactness can be found in \cite[Proposition 2.5 and Remark 2.2]{Chepyzhov_Vishik}.

\begin{proposition}\label{pro:trc}
Let $f\in C(\mathbb{R};C^1(\mathbb{R}))$. Then $f$ is translation compact if and only if for every $R>0$
\begin{itemize}
\item[(i)] $|f(t,u)|+
|\frac{\partial f}{\partial u}(t,u) |\leq C_R$ for $(t,u)\in\mathbb{R}\times [-R,R],$

\item[(ii)]Functions $f(t,u)$ and $\frac{\partial f}{\partial u}(t,u)$ are uniformly continuous on $\mathbb{R}\times[-R,R].$
\end{itemize}
\end{proposition}
We prove two simple results concerning the translation compactness of $f_\varepsilon$ and its hull.
\begin{corollary}
For each $\varepsilon\in(0,1]$ function $f_\varepsilon$ is translation compact.
\end{corollary}
\begin{proof}
From assumption \hyperref[as:H3]{(H3)} and the fact that $f_0\in C^1(\mathbb{R})$ one can deduce that (i) from Proposition \ref{pro:trc} holds. Moreover, \hyperref[as:H6]{(H6)} and Proposition \ref{prop23} imply that  (ii) holds, and the proof is complete.
\end{proof}
\begin{proposition}\label{prop:p}
If $f_\varepsilon$ satisfies conditions \hyperref[as:H1]{(H1)}, \hyperref[as:H2]{(H2)}, \hyperref[as:H3]{(H3)}, and \hyperref[as:H5]{(H5)} then these conditions are satisfied by all elements from $\mathcal{H}_{[0,1]}$. Moreover there exist constants $C,K>0$ independent of $\varepsilon$ such that for every $p_\varepsilon \in \mathcal{H}(f_\varepsilon)$ there hold
\begin{equation}
\sup_{\epsilon\in [0,1]}\sup_{t\in \R}\sup_{u\in \R} \left|p_\varepsilon(t,s)-f_0(u)\right| \leq K,
\qquad
\sup_{\varepsilon\in[0,1]}\sup_{t\in \R}\left|\frac{\partial p_\varepsilon}{\partial u}(t,u)\right| \leq C(1+|u|^{4-\kappa})\ \ \textrm{for every}\ \ u\in \R.
\end{equation}

\end{proposition}
\begin{proof}
Property
\hyperlink{as:H1}{(H1)} is clear.
Suppose that \hyperlink{as:H2}{(H2)} does not hold. Then there exists a number $\delta > 0$, sequences $\varepsilon_n \to 0$, $p_{\varepsilon_n}\in \mathcal{H}(f_{\varepsilon})$, $t_n \in R$ and a number $u\in\mathbb{R}$ such that
$$
|p_{\varepsilon_n}(t_n,u)-f_0(u)|> 2\delta.
$$
Because $p_{\varepsilon_n}\in\mathcal{H}(f_{\varepsilon_n})$
we can pick a sequence $s_n$ such that $|f_{\varepsilon_n}(s_n+t_n,u)-p_{\varepsilon_n}(t_n,u)|
\leq \delta$.
Then
$$
|f_{\varepsilon_n}(t_n+s_n,u)-f_0(u)|\geq -|f_{\varepsilon_n}(t_n+s_n,u)-p_{\varepsilon_n}(t_n,u)|+
|p_{\varepsilon_n}(t_n,u)-f_0(u)|\geq \delta.
$$
Now \hyperlink{as:H2}{(H2)} follows by contradiction.
We denote
$$K_1:=
\sup_{\varepsilon\in[0,1]}\sup_{t\in \R}\sup_{u\in \R} |f_\varepsilon(t,u)-f_0(u)|,
$$
which from assumption \hyperlink{as:H3}{(H3)} is a finite number.
Taking $p_\varepsilon\in \mathcal{H}_{[0,1]}$, for every $(t,u) \in \R^2$ we obtain.
$$
|p_\varepsilon(t,u)-f_0(u)|\leq
|f_\varepsilon(t+s_n,u)-p_\varepsilon(t,u)|+
|f_\varepsilon(t+s_n,u)-f_0(u)|
\leq K_1 + |f_\varepsilon(t+s_n,u)-p_\varepsilon(t,u)|
$$
We can a pick sequence $s_n$ such that
$|f_{\varepsilon}(s_n+t,u)-
p_{\varepsilon}(t,u)|\to 0$. So passing to the limit we get $$
|p_\varepsilon(t,u)-f_0(u)|\leq K_1
$$
We have proved that for every $p_\varepsilon\in \mathcal{H}_{[0,1]}$ we have
$$
\sup_{t\in \R}\sup_{u\in \R}
|p_\varepsilon(t,u)-f_0(u)|\leq K_1.
$$
From \hyperref[as:H5]{(H5)} we obtain
$$
\left|\partu{f_\varepsilon}(t+s_n,u)\right|\leq C(1 + |u|^{4-\kappa}),
$$
for every $u,t,s_n\in\mathbb{R},\varepsilon\in[0,1]$. Again by choosing $s_n$ such that $|f_{\varepsilon}(s_n+t,u)-
p_{\varepsilon}(t,u)|\to 0$ and passing to the limit we observe that 
for every $p_\varepsilon\in\mathcal{H}_{[0,1]}$ there holds
\begin{equation}
\sup_{t\in \R}\left|\frac{\partial p_\varepsilon}{\partial u}(t,u)\right| \leq C(1+|u|^{4-\kappa})\ \ \textrm{for every}\ \ u\in \R,
\end{equation}
which ends the proof.
\end{proof}

\begin{proposition}
	If \hyperref[as:H1]{(H1)}, \hyperref[as:H2]{(H2)}, \hyperref[as:H3]{(H3)}, and \hyperref[as:H5]{(H5)} hold, then for every $R>0$ and every $p_\varepsilon\in \mathcal{H}(f_\varepsilon)$
	$$
	\lim_{\varepsilon\to 0}\sup_{|s|\leq R}\sup_{t\in \R}|p_{\varepsilon}(t,s)-f_0(s)| = 0.
	$$
\end{proposition}

\begin{proof}
	For contradiction assume that there exists $\delta >0$ and sequences $|s_n|\leq R$, $t_n\in \R$, $\varepsilon_n\to 0$ such that
	$$
	\delta\leq |p_{\varepsilon_n}(t_n,s_n)-f_0(s_n)|.
	$$
	For a subsequence there holds $s_n\to s_0$, where $|s_0|\leq R$. Hence,
	\begin{align*}
	& \delta \leq |p_{\varepsilon_n}(t_n,s_n)-p_{\varepsilon_n}(t_n,s_0)| +|p_{\varepsilon_n}(t_n,s_0)-f_0(s_0)|+|f_0(s_0)-f_0(s_n)|\\
	& \leq C(1+|s_n|^{4-\kappa}+|s_0|^{4-\kappa})|s_n-s_0| + \sup_{t\in \R} |p_{\varepsilon_n}(t,s_0)-f_0(s_0)|+|f_0(s_0)-f_0(s_n)|\\
	& \leq C(R)|s_n-s_0| + \sup_{t\in \R} |p_{\varepsilon_n}(t,s_0)-f_0(s_0)|+|f_0(s_0)-f_0(s_n)|.
	\end{align*}
	All terms on right-hand side tend to zero as $n\to \infty$, and we have the contradiction.
\end{proof}

\section{Galerkin solutions.}\label{sec:4}
\begin{definition}\label{def:problem}
Let $(u_0,u_1)\in \mathcal{E}_0$. The function $u\in L^\infty_{loc}([0,\infty); H^1_0)$ with $u_t\in L^\infty_{loc}([0,\infty);L^2)$ and $u_{tt}\in L^\infty_{loc}([0,\infty);H^{-1})$ is a weak solution of problem \eqref{eq:prblm} if for every $v\in L^2_{loc}([0,\infty);H^1_0)$
there holds
$$
\int_{0}^{t_1} \langle u_{tt}(t), v(t)\rangle_{H^{-1}\times H^1_0} + (u_t(t) - f_{\varepsilon}(t,u(t)),v) + (\nabla u(t), \nabla v(t)) \, dt = 0,
$$
and $u(0) = u_0$, $u_t(0) = u_1$.
\end{definition}
Note that as $u\in C([0,\infty);L^2)$ and $u_t\in C([0,\infty);H^{-1})$, pointwise values of $u$ and $u_t$, and thus the initial data, make sense. However, due to the lack of regularity of the nonlinear term $f_{\varepsilon}(\cdot ,u(\cdot))$, we cannot test the equation with $u_t$. Thus, although it is straightforward to prove (using the Galerkin method) the existence of the weak solution given by the above definition, we cannot establish the energy estimates required to work with this solution.

Let $\{e_i \}_{i=1}^\infty$ be the eigenfunctions of the $-\Delta$ operator with the Dirichlet boundary conditions on $\partial \Omega$ sorted by the nondecreasing eigenvalues. They constitute the  orthonormal basis of $L^2$ and they are orthogonal in $H^1_0$. Denote $V_N = \textrm{span}\, \{ e_1,\ldots, e_N \}$. The family of finite dimensional spaces $\{ V_N\}_{N=1}^\infty$ approximates $H^1_0$ from the inside, that is
$$
\overline{\bigcup_{N=1}^\infty V_N}^{H^1_0} = H^1_0\qquad \textrm{and}\qquad V_N\subset V_{N+1}\ \ \textrm{for every}\ \ N\geq 1.
$$
Let $u^N_0\in V_N$ and $u^N_1\in V_N$ be such that
\begin{align*}
    & u^N_0 \to u_0\ \  \textrm{in}\ \ H^1_0\ \ \textrm{as}\ \ N\to \infty,\\
    & u^N_1 \to u_1\ \  \textrm{in}\ \ L^2\ \ \textrm{as}\ \ N\to \infty.
\end{align*}
Now the $N$-th Galerkin approximate solution for \eqref{eq:prblm} is defined as follows.
\begin{definition}\label{def:galerkin}
The function $u^N\in C^1([0,\infty); V_N)$ with $u^N_t\in AC([0,\infty);V_N)$ is the $N$-th Galerkin approximate solution of problem \eqref{eq:prblm} if $u_N(0) = u^N_0$, $u^N_t(0) = u^N_1$ and for every $v\in V_N$ and a.e. $t>0$
there holds
$$
(u^N_{tt}(t)+ u^N_t(t) - f_{\varepsilon}(t,u^N(t)),v) + (\nabla u^N(t), \nabla v) = 0.
$$
\end{definition}
We continue by defining the weak solution of the Galerkin type
\begin{definition}
The weak solution given by Definition \ref{def:problem} is said to be of the Galerkin type if it can be approximated by the solutions of the Galerkin problems, i.e., for a nonrenumbered subsequence of $N$ there holds
\begin{align}
    & u^N \to u\ \ \textrm{weakly-* in}\ \ L^\infty_{loc}([0,\infty);H^1_0),\\
    & u^N_t \to u_t\ \ \textrm{weakly-* in}\ \ L^\infty_{loc}([0,\infty);L^2),\\
    & u^N_{tt} \to u_{tt}\ \ \textrm{weakly-* in}\ \ L^\infty_{loc}([0,\infty);H^{-1}).
\end{align}
\end{definition}
We skip the proof of the following result which is standard in the framework of the Galerkin method.
\begin{theorem}
Assume \hyperlink{as:H1}{(H1)}, \hyperlink{as:H3}{(H3)}--\hyperlink{as:H5}{(H5)}. If $(u_0,u_1)\in \mathcal{E}_0$ then the problem given in Definition \ref{def:problem} has at least one weak solution of Galerkin type.
\end{theorem}

\begin{proposition}\label{prop:abball}
The Galerkin solutions of problem \eqref{eq:prblm} are bounded in $\mathcal{E}_0$ and there exists a bounded  set $B_0\subset \mathcal{E}_0$ which is absorbing, i.e. for every bounded set $B\subset\mathcal{E}_0$ there exists $t_0\geq 0$ such that for every Galerkin solution $(u(t),u_t(t))$  with the initial conditions in $B$ there holds $(u(t),u_t(t)) \in B_0$ for every $t \geq t_0$. Moreover $B_0$ and $t_1$  do not depend on the choice of $p(t,u)\in\mathcal{H}_{[0,1]}$ in place of $f_\varepsilon$  in  \eqref{eq:prblm}.

\end{proposition}

Note that in the above result the function $f_\varepsilon$  in  \eqref{eq:prblm} is replaced by $p\in \mathcal{H}_{[0,1]}$. In the sequel we will consider \eqref{eq:prblm} with such $p\in \mathcal{H}_{[0,1]}$ replacing $f_\varepsilon$. 

 To prove the above proposition we will need the following Gronwall type lemma.
\begin{lemma}\label{le:gron}
Let $I(t)=I_1(t)+\ldots+I_n(t)$ be an absolutely continuous function, $I(0)\in\R$. Suppose that
$$\dt I(t)\leq -A_i I_i(t)^{\alpha_i} + B_i,$$
for every $i\in\set{1,\ldots,n}$ and for almost every $t$ such that $I_i(t)\geq 0$, where $\alpha_i, A_i, B_i > 0$ are constants.
Then for every $\eta>0$ there exists $t_0>0$ such that
$$
I(t)\leq
\sum_{i=1}^n\left(\frac{B_i}{A_i}\right)^{\frac{1}{\alpha_i}}
+
\eta,\;
\text{for every $t\geq t_0$}.
$$
If, in addition, $\{I^l(t)\}_{l \in\mathcal{L}}$ is a family of functions satisfying the above conditions and such that $I^l(0)\leq Q$ for each $l\in \mathcal{L}$, then the time $t_0$ is independent of $l$ and there exists a constant $C$ depending on $Q, A_i, B_i, \alpha_i$ such that $I^l(t)\leq C$ for every $t\geq 0$ and every $l\in \mathcal{L}$. 
\end{lemma}
\begin{proof}
We denote $$B=\sum_{i=1}^n\left(\frac{B_i}{A_i}\right)^{\frac{1}{\alpha_i}}$$ and let $A= \min_{i\in \{ 1,\ldots,n\}} \{ A_i  \}$.
First we will show that for every $\eta>0$
if $I(t_0)\leq B+\eta$, then
$I(t)\leq B+\eta$ for every $t\geq t_0$.
For the sake of contradiction let us suppose that
there exists some $t_1>t_0$ such that $I(t_1)>B+\eta$.
Let $t_2=\sup \set{ s\in[t_0,t_1]\,:\ 
 I(s)\leq B + \eta
}$,
so there exists $\delta>0$ such that for every $s\in (t_2,t_1]$ we can find an index $i$ for which there holds $$I_i(s)>\left(\frac{B_i}{A_i}+\delta\right)^{\frac{1}{\alpha_i}}.$$ Then for a.e. $s\in (t_2,t_1]$ we have
$$
\dt I(s)\leq -A_iI_i(s)^{\alpha_i}+B_i\leq -A\delta,
$$
and
after integration we get that $I(t_1)<I(t_2) -(t_2 - t_1) A\delta$ which is a contradiction. We observe that all functions from family $\{I^l(t)\}_{l\in\mathcal{L}}$ are bounded by  $\max\{Q,1\} + B$.
Now we will prove existence of $t_0$. For the sake of contradiction suppose that there 
exists $\eta>0$ and the sequence of times $t_n \to \infty$ such that $I^{l_n}(t_n)>B+\eta$ for some $l_n\in\mathcal{L}$. Then for every $s\in [0,t_n]$ we must have $I^{l_n}(s)>B+\eta$. Then there exist $\delta>0$ such for all $s\in [0,t_n]$ and $l_n$ there is some $I_i^{l_n}(s)>(\frac{B_i}{A_i}+\delta)^\frac{1}{\alpha_i}$.
Again for a.e $s\in (0,t_n)$
$$\dt I^{l^n}(s)\leq -A\delta$$
and after integrating we get that
$I^{l_n}(t_n) \leq Q - t_nA\delta$
which is contradiction.
\end{proof}
\begin{proof}[Proof of Proposition \ref{prop:abball}]
Let $u$ be the Galerkin solution to \eqref{eq:prblm} with any function  
$p\in\mathcal{H}_{[0,1]}$ in place of $f_\varepsilon$ at the right-hand side of \eqref{eq:prblm}. By testing this equation with $u+2\ut$ we obtain
\begin{align*}
 &   \dt\left[
(\ut,u)+\frac{1}{2}\norm{u}^2+
\norm{\ut}^2+\norm{\nabla u}^2-
2\int_\Omega F_0(u)dx
\right]\\
& \qquad 
=-\norm{\ut}^2-\norm{\nabla u}^2 +(f_0(u),u)
+
(p(t,u) - f_0(u),2 \ut+u)
,
\end{align*}
where $F_0(u)=\int_0^uf_0(v)dv$. Assumption \hyperref[as:H4]{(H4)} implies the inequality
$$(f_0(u),u)\leq C+K\norm{u}^2,\text{ where }0\leq K<\lambda_1.$$
We define $$I(t)=
(\ut,u)+\frac{1}{2}\norm{u}^2+
\norm{\ut}^2+\norm{\nabla u}^2-
2\int_\Omega F_0(v)dx.$$
Using the Poincar\'{e} and Cauchy--Schwarz inequalities we obtain
\begin{equation}
\dt I(t) \leq - \norm{u_t}^2 -C \norm{\nabla u}^2 + \norm{p_{\varepsilon}(t,u) - f_0(u)}(2\norm{ \ut}+ \norm{u})+ C.
\end{equation}
Using the Poincar\'{e} inequality again it follows by Proposition \ref{prop:p} that 
\begin{equation}\label{in:estimate1}
\dt I(t) \leq - C\left(\norm{u_t}^2 +\norm{\nabla u}^2\right) + C.
\end{equation}

\noindent
We represent the function $I(t)$ as the sum of the following terms
\begin{equation*}
 I_1 = \norm{\ut}^2,\;I_2=\frac{1}{2}\norm{ u}^2,\;I_3=\norm{\nabla
 u}^2,\;
 I_4(t)=(\ut,u),\;
 I_5=-2\int_{\Omega}F_0(u)dx.
\end{equation*}
From the estimate \eqref{in:estimate1} and Poincare inequality we can easily see that
\begin{equation}
    \text{$\dt I \leq -A_iI_i +B_i$ for $i\in\{1,2,3,4\}$,}
\end{equation}
where $A_i, B_i$ are positive constants. 
To deal with the term $I_5$ we observe that by the growth condition \hyperref[as:H5]{(H5)} using the H\"{o}lder
inequality we obtain
\begin{align*}
    I_5&\leq C \int_\Omega\left|\int_0^u 1+|v|^5dv\right| dx\leq C\int_\Omega\left(|u|+|u|^6\right) dx= C \left(\norm{u}_{L_1}+\norm{u}_{L_6}^6\right)\leq C\left(\norm{u}_{L_6}+\norm{u}_{L_6}^6\right)
    \\
    &\leq
     C\left(\norm{u}_{L_6}^6 + 1\right).
\end{align*}
From the Sobolev embedding  $H_0^1\hookrightarrow L_6$ it follows that 
\begin{equation}
     I_5^{\frac{1}{3}}
     \leq\left(C\norm{\nabla u}^6+ 1\right)^{\frac{1}{3}}
     \leq C\left(\norm{\nabla u}^2 +  1\right).
\end{equation}
From the estimate \eqref{in:estimate1} we observe that
\begin{equation}
    \dt I \leq -A_5I_5^{\frac{1}{3}} + B_5, \ \ \textrm{with}\ \ A_5,B_5 > 0.
\end{equation}
By Lemma \ref{le:gron} we deduce that there exists a constant $D>0$ such that every for bounded set of initial data $B\subset\mathcal{E}_0$ there exists the time $t_0=t_0(B)$ such that for every $p\in\mathcal{H}_{[0,1]}$ there holds
\begin{equation}
    \text{$I(t)\leq D$ for $t\geq t_0$ and $(u_0,u_1)\in B$.}
\end{equation}
We observe that from  \hyperref[as:H4]{(H4)} it follows that
\begin{equation}
    \text{$F_0(u)\leq C + \frac{K}{2}u^2$ where $0\leq K<\lambda_1$ }
\end{equation}
We deduce
\begin{equation}
    I(t)\geq  \frac{1}{2}\norm{\ut}^2 + \norm{\nabla u}^2 - K\norm{u}^2 - C\geq C\norm{u}_{\mathcal{E}_0}^2 - C.
\end{equation}
We have shown the existence of the absorbing set $B_0\subset \mathcal{E}_0$ which is independent of the choice of $p\in\mathcal{H}_{[0,1]}$. By Lemma $\ref{le:gron}$ it follows that for every initial condition $(u_0,u_1)\in \mathcal{E}_0$ there exists a constant $D=D(u_0,u_1) > 0$ such that for every $p\in\mathcal{H}_{[0,1]}$ and $t\in\mathbb{R}$ there holds
\begin{equation}
    \text{$I(t)\leq D$ for $t\in [0,\infty)$. }
\end{equation}
The proof is complete.
\end{proof}

\section{Shatah--Struwe solutions, their regularity and a priori estimates.}\label{sec:5}

\subsection{Auxiliary linear problem.} Similar as in \cite{Savostianov} we define an auxiliary  non-autonomous problem for which we derive a priori estimates both in energy and Strichartz norms.
\begin{align}
\begin{cases}
& u_{tt}+u_t-\Delta u = G(x,t)\; \text{for}\; (x,t)\in \Omega\times(t_0,\infty),\label{eq:linear}\\
&u(t,x) = 0\; \text{for}\; x\in \partial \Omega
\\&u(t_0,x) = u_0(x)
\\&u_t(t_0,x) = u_1(x)
\end{cases}
\end{align}
It is well known that if only $G\in L^1_{loc}([t_0,\infty);L^2)$ and $(u_0,u_1)\in \mathcal{E}_0$ then the above problem has the unique weak solution $u$ belonging to $C_{loc}([t_0,\infty);H^1_0)$ with $u_t\in C_{loc}([t_0,\infty);L^2)$ and $u_{tt}\in L^\infty_{loc}([t_0,\infty);H^{-1})$. For details see cf. \cites{Temam, Babin_Vishik, Robinson, Chepyzhov_Vishik}.
The next result appears in \cite[Proposition 2.1]{Savostianov}. For completeness of our argument we provide the outline of the proof.
\begin{proposition}\label{prop:linearestimate}
Let $u$ be the weak solution to problem \eqref{eq:linear} on interval $[t_0,\infty)$ with $G\in L^1_{loc}([t_0,\infty);L^2)$ and initial data $u(t_0) = u_0$, $u_t(t_0) = u_1$ with $(u_0,u_1) \in \mathcal{E}_0$. Then the following estimate holds
$$\norm{(u(t),u_t(t))}_{\mathcal{E}_0}\leq C \left(\norm{(u_0,u_1)}_{\mathcal{E}_0}e^{-\alpha (t-t_0)}+\int_{t_0}^t e^{-\alpha(t-s)}\norm{G(s)}ds\right)$$
for every $t\geq t_0$, where $C,\alpha$ are positive constants independent of $t,t_0,G$ and initial conditions of \eqref{eq:linear}.
\end{proposition}
\begin{proof}
Testing \eqref{eq:linear} by $u+2u_t$ we obtain
\begin{align*}
    \dt\left(
(\ut,u)+\frac{1}{2}\norm{u}^2+
\norm{\ut}^2+\norm{\nabla u}^2\right)
=-\norm{\ut}^2-\norm{\nabla u}^2 +(G(t),u+2\ut)
\end{align*}
We define $I(t) = (\ut,u)+\frac{1}{2}\norm{u}^2+
\norm{\ut}^2+\norm{\nabla u}^2$. We easily deduce
$$
\dt I(t)\leq C\left( -I(t) + \sqrt{I(t)}\norm{G(t)}\right).
$$
Multiplying the above inequality by $e^{Ct}$ we obtain
$$
\frac{d}{dt}\left(I(t)e^{Ct}\right) \leq C e^{Ct}\|G(t)\|\sqrt{I(t)}.
$$
After integration it follows that
$$
I(t)e^{Ct}-I(t_0)e^{Ct_0} \leq
C\int_{t_0}^t e^{Cs}\|G(s)\| \sqrt{I(s)}\, ds.
$$
Hence, for every $\varepsilon > 0$
\begin{equation}\label{in:estimate2}
I(t) \leq (I(t_0)+\varepsilon)e^{C(t_0-t)} +  e^{-Ct}C\int_{t_0}^t e^{Cs}\|G(s)\| \sqrt{I(s)}\, ds.
\end{equation}
Now let
$$
J(t) = C\int_{t_0}^t e^{Cs}\|G(s)\| \sqrt{I(s)}\, ds.
$$
Then $J$ is absolutely continuous, $J(t_0) = 0$, and for almost  every $t > t_0$ we obtain
$$
J'(t) = Ce^{Ct}\|G(t)\|\sqrt{I(t)}.
$$
From \eqref{in:estimate2} it follows that
$$
J'(t) \leq  C e^{Ct}\|G(t)\|\sqrt{(I(t_0)+\varepsilon)e^{C(t_0-t)}+e^{-Ct}J(t)} = C e^{\frac{Ct}{2}}\|G(t)\|\sqrt{(I(t_0)+\varepsilon)e^{Ct_0}+J(t)}.
$$
Hence
$$
\frac{J'(t)}{\sqrt{(I(t_0)+\varepsilon)e^{Ct_0}+J(t)}} \leq Ce^{\frac{Ct}{2}}\|G(t)\|.
$$
After integrating over interval $[t_0,t]$ we obtain the following inequality valid for every $t\geq t_0$
$$
\sqrt{(I(t_0)+\varepsilon)e^{Ct_0} + J(t)}  \leq
{\sqrt{(I(t_0)+\varepsilon)e^{Ct_0}}}+
\frac{C}{2}\int_{t_0}^t e^{\frac{Cs}{2}}\|G(s)\|\, ds.
$$
It follows that
$$
J(t)\leq C\left[
\left(\int_{t_0}^t e^{\frac{Cs}{2}}\|G(s)\|\, ds\right)^2+(I(t_0)+\varepsilon)e^{Ct_0}\right].
$$
From definition of $J(t)$ using the  inequality \eqref{in:estimate2} we notice that
$$
I(t)\leq C \left((I(t_0)+\varepsilon)e^{\alpha(t_0-t)}+
\left(\int_{t_0}^t e^{-\alpha(t - s)} \|G(s)\|\, ds\right)^2\right),
$$
for a constant $\alpha>0$.
As
$c_1\norm{(u(t),u_t(t))}_{\mathcal{E}_0}\leq
\sqrt{I(t)}\leq
c_2\norm{(u(t),u_t(t))}_{\mathcal{E}_0}$ for some $c_1,c_2>0$, passing with $\varepsilon$ to zero we obtain the required assertion.
\end{proof}
The following Lemma provides us an extra control on the $L^4(L^{12})$ norm of the solution to the linear problem \eqref{eq:prblm}. The result is given in  \cite[Proposition 2.2 and Remark 2.3]{Savostianov}.
\begin{lemma}\label{lemma:strichartz}
Let $h>0$ and let $u$ be a weak solution to problem \eqref{eq:linear} on time interval $(t_0,t_0+h)$ with $G\in L^1(t_0,t_0+h;L^2)$ and $(u(t_0),u_t(t_0)) = (u_0,u_1)\in \mathcal{E}_0$. Then $u\in L^4(t_0,t_0+h;L^{12})$ and the following estimate holds
\begin{equation}\label{eq:strichartz}
\norm{u}_{L^4(t_0,t_0+h;L^{12})}
\leq
C_h\left( \norm{(u_0,u_1)}_{\mathcal{E}_0}
+
\norm{G}_{L^1(t_0,t_0+h;L^2)} \right),
\end{equation}
where the constant $C_h > 0$ depends only on $h$ but is independent of $t_0,(u_0,u_1),G$.
\end{lemma}
We will need the following result.
\begin{proposition}
It is possible to choose the constants $C_h$ in previous lemma such that the function $[0,\infty) \ni h \to C_h$ is nondecreasing.
\end{proposition}
The above proposition will be proved with the use of the following theorem known as the Christ--Kiselev lemma, see e.g. \cite[Lemma 3.1]{Sogge}.
\begin{theorem}\label{th:kiselev}
  Let $X,Y$ be Banach spaces and assume that $K(t,s)$ is a continuous function taking values in $B(X,Y)$, the space of linear bounded mappings from $X$ to $Y$. Suppose that $-\infty \leq a < b \leq \infty $ and set
  $$ Tf(t) = \int_a^b K(t,s) f(s)\, ds,$$
  $$ Wf(t) = \int_a^t K(t,s) f(s)\, ds.$$
  Then if for $1\leq p<q\leq \infty$ there holds
  $$\norm{Tf}_{L^q(a,b;Y)} \leq C \norm{f}_{L^p(a,b;X)} ,$$
  then
  $$\norm{Wf}_{L^q(a,b;Y)} \leq \overline{C}
  \norm{f}_{L^p(a,b;X)},\;
  \text{with $\overline{C}=2C\frac{2^{2\left(\frac{1}{q} - \frac{1}{p} \right)} }{1 - 2^{\frac{1}{q} -\frac{1}{p} } } $}.$$
\end{theorem}
\begin{proof}
If $G\equiv 0$ then we denote the corresponding constant by $D_h$, i.e. 
$$
\norm{u}_{L^4(t_0,t_0+h;L^{12})}
\leq
D_h \norm{(u_0,u_1)}_{\mathcal{E}_0}.
$$
Clearly, the function $[0,\infty) \ni h\to D_h \in [0,\infty)$ can be made nondecreasing. We will prove that \eqref{eq:strichartz} holds with $C_h$, a monotone function of $D_h$. If the family $\{S(t)\}_{t\in \R}$ of mappings $S(t):\mathcal{E}_0\to \mathcal{E}_0$ is the solution group for the linear homogeneous problem (i.e. if $G\equiv 0$) then we denote  $S(t)(u_0,u_1) = 
(S_u(t)(u_0,u_1),S_{u_t}(t)(u_0,u_1))$. Let $t_0\in \R$ and $\delta > 0$. Using the Duhamel formula for equation $\eqref{eq:linear} $ we obtain
$$u(t_0+\delta) = S_u(\delta)(u_0,u_1) + \int_{0}^{\delta}
S_u(\delta - s)(0,G(t_0+s))\, ds.$$
Applying the $L^4(0,h;L^{12})$ norm with respect to $\delta$  to both sides we obtain
$$
\norm{u}_{L^4(t_0,t_0+h;L^{12}) }\leq
D_{h}\norm{(u_0,u_1)}_\mathcal{E}+
\norm{P_1}_{L^4(0,h;L^{12})},
$$
for every $h>0$, 
where
$
P_1(\delta) = \int_{0}^{\delta}
S_u(\delta - s)(0,G(t_0+s))ds$.
We will estimate the Strichartz norm of $P_1$ using Theorem \ref{th:kiselev} with $X=L^{2},Y=L^{12},q=4,p=1,a=0,b=h$. If $\Pi_N:L^2\to V_N$ is $L^2$-orthogonal projection, then 
$S_u(h-s)(0,\Pi_N(\cdot))$ is a continuous function of $(h,s)$ taking its values in $B(L^2,L^{12})$. Hence the estimate should be derived separately for every $N$, and, since it is uniform with respect to this $N$ it holds also in the limit. We skip this technicality and proceed with the formal estimates only.  We set
$P_2(\delta) = \int_{0}^{h}
S_u(\delta - s)(0,G(t_0+s))ds$,
and we estimate
\begin{align*}
& \norm{P_2}_{L^4(0,h;L^{12})} \leq
\int_{0}^{h}
\norm{ S_u(\delta - s)(0,G(t_0+s))}_{L^4(0,h;L^{12})}\,ds\\
&\ \ \  =
\int_{0}^{h}
\norm{
S_u(\delta)S(-s)(0,G(t_0+s))
}_{L^4(0,h;L^{12})}\, ds \leq
\int_{0}^{h}
D_{h}\norm{S(-s)(0,G(t_0+s))}_{\mathcal{E}_0}
\, ds,
\end{align*}
where in the last inequality we used the homogeneous Strichartz estimate.
Observe that there exists $\beta>0$ such that there holds
$$\norm{S(-s)(u_0,u_1)}_{\mathcal{E}_0} 
\leq 
e^{s \beta}\norm{(u_0,u_1)}_{\mathcal{E}_0}.
$$
We deduce 
$$
\norm{P_2}_{L^4(0,h;L^{12})} \leq D_{h} e^{\beta h} \norm{G}_{L^1(t_0,t_0+h,L^2)}
.
$$
Hence, by Theorem \ref{th:kiselev} we obtain
$
\norm{P_1}_{L^4(0,h;L^{12})} 
\leq
C D_{h} e^{\beta h} \norm{G}_{L^1(t_0,t_0+h,L^2)}$ for every $h > 0$, and the proof is complete.

\end{proof}
The following result will be useful in the bootstrap argument on the attractor regularity.
\begin{lemma}\label{le:fractionlstrichartz}
Let $(u_0,u_1)\in \mathcal{E}_s$ and $G \in L^1_{loc}([t_0,\infty);\mathbb{H}^s)$ for $s\in (0,1]$. Then the weak solution of  \eqref{eq:linear} has regularity $u \in C_{loc}([t_0,\infty);\mathbb{H}^{s+1})$ and $u_t\in C_{loc}([t_0,\infty);\mathbb{H}^s)$, Moreover, the following estimates hold
\begin{align*}
 &   \|(u(t),u_t(t))\|_{\mathcal{E}_s} \leq C\left(\norm{(u_0,u_1)}_{\mathcal{E}_s}e^{-\alpha(t-t_0)} + \int_{t_0}^t e^{-\alpha(t-s)}\norm{G(s)}_{\mathbb{H}^s} ds\right),\\
& \|u\|_{L^4(0,h;W^{s,12})} \leq C_h\left(\norm{(u_0,u_1)}_{\mathcal{E}_s} +
\norm{G}_{L^1(t_0,t_0+h;\mathbb{H}^{s})}\right).
\end{align*}

\end{lemma}
\begin{proof}
The problem
\begin{align}
\begin{cases}
& w_{tt}(t)+w_t(t)-\Delta w(t) = (-\Delta)^{s/2}G(t)\; \text{for}\; (x,t)\in \Omega\times(t_0,\infty),\\
&w(t,x) = 0\; \text{for}\; x\in \partial \Omega ,
\\&w(t_0) = (-\Delta)^{s/2}u_0,
\\&w_t(t_0) = (-\Delta)^{s/2}u_1,
\end{cases}
\end{align}
has the unique weak solution $w\in C_{loc}([t_0,\infty);H^1_0)$ with the derivative $w_t\in C_{loc}([t_0,\infty);L^2)$. It is enough to observe that $$\widehat{w}_k(t) =\lambda_k^{\frac{s}{2}} \widehat{u}_k(t)\ \ \textrm{for every}\ \ k \in \mathbb{N}.$$ Testing weak solutions $w,u$ with $e_kv(t)$, where $v(t)\in C_0^{\infty}([t_0,t_1))$ and using the du Bois-Reymond lemma 
we get systems
$$
\begin{cases}
\widehat{u}_k''+\widehat{u}_k''+\lambda_k \widehat{u}_k = (G(t),e_k),
\\\widehat{u}_k(t_0)= \widehat{u_0}_k,
\\\widehat{u}_k'(t_0)= \widehat{u_1}_k,
\end{cases}
\ \ \ 
\begin{cases}
\widehat{w}_k''+\widehat{w}_k'+\lambda_k \widehat{w}_k = (\fraclap{s}{2}G(t),e_k) = \lambda_k^{\frac{s}{2}}(G(t),e_k),
\\\widehat{w}_k (t_0)=  (\fraclap{s}{2}u_0, e_k) = \lambda_k^{\frac{2}{s}}\widehat{u_0}_k,
\\\widehat{w}_k'(t_0)= (\fraclap{s}{2} u_1,e_k)=
\lambda_k^{\frac{2}{s}}\widehat{u_1}_k.
\end{cases}
$$
The difference $\overline{w}_k(t)=\widehat{w}_k(t) -\lambda_k^{\frac{s}{2}} \widehat{u}_k(t)$ solves the problem
$$
\begin{cases}
\overline{w}_k''+\overline{w}_k'+\lambda_k\overline{w}_k = 0,\\
\overline{w}_k(t_0) = 0,\\
\overline{w}_k'(t_0) = 0.
\end{cases}
$$
So $\overline{w}_k(t)=0$ for every $t\in [t_0,\infty)$. The assertion follows from Proposition \ref{prop:linearestimate} and Lemma \ref{lemma:strichartz}.
\end{proof}

\subsection{Shatah--Struwe solutions and their properties}

This section recollects the results from \cite{Savostianov}. The non-autonomous generalizations of these results are straightforward so we skip some of the proofs which follow the lines of the corresponding results from \cite{Savostianov}. 
The following remark follows from the Gagliardo--Nirenberg interpolation inequality and the Sobolev embedding $H^1_0\hookrightarrow L^6$.
\begin{remark}
If $u\in L^4(0,t;L^{12})$ and $u\in L^\infty(0,t;H^1_0)$ then
$$
\norm{u}_{L^5(0,t;L^{10})}
\leq \norm{u}_{L^4(0,t;L^{12})}^{\frac{4}{5}} 
\norm{u}_{L^\infty(0,t;H^1_0)}^\frac{1}{5} .
$$
\end{remark}
We define the Shatah--Struwe solution of problem \eqref{eq:prblm}.
\begin{definition}\label{def:shatah}Let $(u_0,u_1)\in \mathcal{E}_0$. A weak solution of problem \eqref{eq:prblm}, given by Definition \ref{def:problem} is called a Shatah--Struwe solution if $u\in L^4_{loc}([0,\infty);L^{12})$.
\end{definition}

\begin{proposition}\label{prop:uniq}
Shatah--Struwe solutions to problem \eqref{eq:prblm} given by Definition \ref{def:shatah} are unique and the mapping $\mathcal{E}_0 \ni (u_0,u_1) \mapsto (u(t),u_t(t))\in \mathcal{E}_0$ is continuous for every $t>0$.
\end{proposition}
\begin{proof}
Let $u,v$ be Shatah--Struwe solutions to Problem \eqref{eq:prblm} with the initial data $(u_0,u_1)$ and $(v_0,v_1)$, respectively.  Their difference $w:=u-v$ satisfies the following equation
$$w_{tt}(t)+w_t(t)-\Delta w(t) = f_\varepsilon(t,u(t))-f_\varepsilon(t,v(t))=w\frac{\partial p_{\varepsilon}(t,\theta u + (1-\theta)v)}{\partial u} .$$
Testing this equation with $w_t$ yields
$$
\frac{1}{2}\dt
\left(\norm{w_{t}}^2+\norm{\nabla w}^2\right)+ \norm{w_{t}}^2 =\left(w\frac{\partial p_\varepsilon(t,\theta u + (1-\theta)v)}{\partial u},w_t\right).
$$
Assumption \hyperref[as:H5]{(H5)} gives inequality

$$
\frac{1}{2}\dt\left(\norm{w_{t}}^2+\norm{\nabla w}^2\right)
\leq
C\int_{\Omega}w(1+|u|^4+|v|^4)w_tdx
$$
Then by using the H{\"o}lder inequality with exponents $\frac{1}{6},\frac{1}{3},\frac{1}{2}$ and the Sobolev embedding $L_6\hookrightarrow H_0^1$ we obtain
$$
\dt\left(\norm{w_{t}}^2+\norm{\nabla w}^2\right)
\leq
C\left(\norm{\nabla w}^2+\norm{w_t}^2\right)\left(1+
\norm{u}_{L^{12}}^{4}+\norm{v}_{L^{12}}^4\right).
$$
Because $v,u$ are Shatah–Struwe solution, i.e. $u,v\in L^4_{loc}([0,\infty);L^{12})$, it is possible to use integral form of the Grönwall inequality which gives us
$$\norm{\nabla w}^2+\norm{w_t}^2
\leq
(\norm{\nabla w_0}^2+ \norm{w_1}^2) 
\exp\left(C\left(t +\int_{0}^t\norm{v}_{L_{12}}^{4}+\norm{w}_{L_{12}}^{4}dt\right)\right),
$$
for $t\in[0,\infty)$, hence the assertion follows.
\end{proof}

\begin{lemma}\label{lem:savostianov}
Every weak solution of problem \eqref{eq:prblm} is of Galerkin type if and only it is a Shatah--Struwe solution. Moreover for every $t>0$ there exists a constant $C_t>0$ such that for every solution $u$ with arbitrary $p_\varepsilon \in\mathcal{H}_{[0,1]}$ treated as right-hand side in \eqref{eq:prblm} contained  in the absorbing set $B_0$ there holds
$$
\norm{u}_{L^4(0,t;L^{12} )}\leq C_t.
$$
\end{lemma}
\begin{proof}
Let $u$ be the solution of the Galerkin type with the initial data $(u_0,u_1)\in \mathcal{E}_0$. From assumption \hyperref[as:H3]{(H5)} we see that
$$\norm{p_{\varepsilon}(t,u)}\leq  C(1+\norm{|u|^{5-\kappa}})= C(1+\norm{u}^{5-\kappa}_{L^{2(5-\kappa)}} )
\leq
C(1+ \norm{u}^{5-\kappa}_{L^{10}}) .$$
We assume that $t\in[0,1]$.
From the H{\"o}lder inequality we obtain
\begin{align*}
  &  \int_{0}^{t}\norm{p_{\varepsilon}(s,u)}ds
\leq
Ct +C\int_{0}^{t}\norm{u}^{5-\kappa}_{L^{10}}ds
\leq
C\left(\left(\int_{0}^{t}\norm{u}^{5}_{L^{10}}ds\right)^{\frac{5-\kappa}{5}}
\left(\int_{0}^{t}1dt\right)^{\frac{\kappa}{5}} + t\right)\\
&\ \ =
C \left(\norm{u}_{L^5(0,t; L^{10})}^{5-\kappa }t^{\frac{\kappa}{5}} + t \right)
\leq
C\left(\norm{u}_{L^5(0,t;L^{10}) }^{5 - \kappa} + 1\right)t^{\frac{\kappa}{5}}
\\
&\ \ \leq
CR^{\frac{1}{5}}\left(\norm{u}_{L^4(0,t;L^{12})}^{4 - \frac{4\kappa}{5}} + 1 \right)t^{\frac{\kappa}{5}}
\end{align*}
where $R$ is the bound of the $L^\infty(0,t;H^1_0)$ norm of $u$.
We split $u$ as the sum 
$ u = v+ w$ where $v,w$  solve the following problems
$$
\begin{cases}
v_{tt}+v_t-\Delta v = 0,
\\v(t,x) = 0\; \text{for}\; x\in \partial \Omega,
\\v(0,x) = u_0(x),
\\v_t(0,x) = u_1(x),
\end{cases}
\qquad\qquad
\begin{cases}
w_{tt}+w_t-\Delta w = p_\epsilon(t,u),
\\w(t,x) = 0\; \text{for}\; x\in \partial \Omega,
\\w(0,x) = 0,
\\w_t(0,x) = 0.
\end{cases}
$$
From the Strichartz estimate in Lemma \ref{lemma:strichartz} we deduce
$$
\norm{v}_{L^4(0,t;L^{12}) } \leq C_1 \norm{(u_0,u_1)}_\mathcal{E},
$$
and
$$
\norm{w}_{L^4(0,t;L^{12}) }\leq
C R^{\frac{1}{5}}
\left(
\norm{w}_{L^4(0,t;L^{12}) }^{4 -\frac{4\kappa}{5}}
+\left(C_1\norm{(u_0,u_1)}\right) ^{4 -\frac{4\kappa}{5}}
+1\right)
t^{\frac{\kappa}{5}}.
$$
We define function $Y(t)= \norm{w}_{L^4(0,t;L^{12})}$ for $t\in[0,1]$.
Formally we do not know if this function is well defined, so to make the proof rigorous we should proceed for Galerkin approximation, cf. \cite{Savostianov}. We continue the proof in formal way. The function  $Y(t)=\norm{w}_{L^4(0,t;L^{12})}$ is continuous with  $Y(0)=0$ and there holds
$$
Y(t)\leq
CR^{\frac{1}{5}}
(
Y(t)^{4 -\frac{4\kappa}{5}} 
+(C_1\norm{(u_0,u_1)}) ^{4 -\frac{4\kappa}{5}}
+1)
t^{\frac{\kappa}{5}}
.
$$
We define
$$t^{\frac{\kappa}{5}}_{\max}
=
\min\left\{
\frac{1}
{2 C R^\frac{1}{5}(
(C_1 R )^{4-\frac{4\kappa}{5}}+2
)}
,
1
\right\}
,
\text{ where }
R \geq \norm{(u_0,u_1)}_{\mathcal{E}_0}   
$$
Now we will use continuation method to prove that the estimate $Y(t)\leq 1$ holds on
the interval $[0,t_{\max}]$. The argument follows the scheme of the proof from \cite[Proposition 1.21]{Tao}.   Defining the logical predicates
$H(t)=(Y(t)\leq 1)$ and  $C(t)= (Y(t)\leq \frac{1}{2} )$ we
observe that following facts hold
\begin{itemize}
 \item $C(0)$ is true.
 \item If $C(s_0)$ for some $s_0$ is true then $H(s)$ is true in some neighbourhood of  $s_0$.
 \item If $s_n\to s_0$ and $C(s_n)$ holds for every $n$ then $C(s_0)$ is true.
 \item $H(t)$ implies $C(t)$ for  $t\in [0,t_{\textrm{max}}]$, indeed 
 $$
 Y(t)\leq
 CR^\frac{1}{5}((C_1 \norm{u_0,u_1})^{4-\frac{4\kappa}{5}}+1+Y(t)^{4-\frac{4\kappa}{5}} )t^{\frac{\kappa}{5}}
 \leq
 C R^\frac{1}{5} (C_1 \norm{u_0,u_1}+2)^{\frac{4-\kappa}{5}} t^{\frac{\kappa}{5}}_{\max}
 \leq
 \frac{1}{2}.
 $$

\end{itemize}
The continuation argument implies that $C(t)$ holds for
$t\in [0,t_{\max}]$.
From the triangle inequality we conclude that
$$\norm{u}_{L^4(0,t_{\max};L^{12})}\leq C_1\norm{(u_0,u_1)} + 1.$$
Observe that $t_{\max}$ and $C_1$ are independent of choice of $p_\varepsilon\in \mathcal{H}_{[0,1]}$.
Because all trajectories are bounded. cf. Proposition \ref{prop:abball}, by picking $R :=\max_{t\in[0,\infty)} \norm{(u(t),u_t(t))}_{\mathcal{E}_0} $ we 
deduce that 
$\norm{u}_{L^4(0,t,L^{12})}$ 
is bounded for every $t>0$. Moreover if $(u(t),u_t(t)) \in B_0$ for every $t\geq  0$, then with $R := \sup_{(u,v)\in B_0} \norm{(u,v)}_{\mathcal{E}_0}$ we get the  bound  $\norm{u}_{L^4(0,t,L^{12})} \leq C_t$ with $C_t$ independent of $p_\varepsilon$.
\end{proof}
\begin{remark}
As a consequence of Proposition \ref{prop:uniq} and Lemma \ref{lem:savostianov} for every $(u_0,u_1)\in \mathcal{E}_0$  weak solution of Galerkin type of problem \eqref{eq:prblm}  is unique.
\end{remark}
\begin{lemma}\label{lem:cont}
	If the weak solution $(u,u_t)$ of Problem \ref{eq:prblm} is of Galerkin type then for every $T>0$ it belongs to the space $C([0,T];\mathcal{E}_0)$.
\end{lemma}
\begin{proof}
	The proof follows an argument of Proposition 3.3 from \cite{Savostianov}. They key fact is that Galerkin (or equivalently, Shatah--Struwe) solutions satisfy the energy equation. Let $t_n\to t$ and let $T> \sup_{n\in N}\{ t_n \}$. Clearly, $(u,u_t) \in C_w([0,T];\mathcal{E}_0)$ and hence $(u(t_n),u_t(t_n))\to (u(t),u_t(t))$ weakly in $\mathcal{E}_0$. To deduce that this convergences is strong we need to show that $\|(u(t_n),u_t(t_n))\|_{\mathcal{E}_0} \to \|(u(t),u_t(t))\|_{\mathcal{E}_0}$. To this end we will use the energy equation
	$$
	\norm{(u(t),u_t(t))}_{\mathcal{E}_0}^2-
	\norm{(u(t_n),u_t(t_n))}_{\mathcal{E}_0}^2 
	=
	2\int_{t_n}^t (p(s,u(s)),u_t) -\norm{u_t(s)}^2\, ds.
	$$
Then
$$
\left|\norm{(u(t),u_t(t))}_{\mathcal{E}_0}^2-
	\norm{(u(t_n),u_t(t_n))}_{\mathcal{E}_0}^2\right| \leq C R\left((R+1)|t-t_n| + 
	\norm{u}_{L^5{(t_n,t;L^{10})}}\right)
$$	
where $R$ is a bound on $\|u_t\|$. The right side tends to zero as $t_n\to t$ which proves the assertion.
	\end{proof}

\subsection{Nonautonomous dynamical system.} We will denote by
$(u(t),u_t(t)) = \varphi_\varepsilon(t,p)(u_0,u_1)$ the map which gives  the solution of  \eqref{eq:prblm} with $p\in\mathcal{H}(f_\varepsilon)$ as the right-hand side  and the initial conditions
$u(0)=u_0$, $u(0)=u_1$.

\begin{proposition}\label{eq:propcont}
 Mapping $\varphi_\varepsilon:\mathbb{R}\times\mathcal{H}(f_\varepsilon)\to C(\mathcal{E})$ together with time translation $\theta_t p_\varepsilon = p_\varepsilon (\cdot + t)$ forms a NDS.
\end{proposition}
\begin{proof}
  Property $\varphi(0,p) = \textrm{Id}_{\mathcal{E}_0}$ and cocycle property are obvious from definition of $\varphi_\varepsilon$ and $\theta_t$. Let $(u^n_0,u^n_1)\to (u_0,u_1)$ in $\mathcal{E}_0$, $p_\varepsilon^n\to p_\varepsilon$ in the metric of $\Sigma$, $t_n\to t$ and let $\{u^n\}_{n=1}^\infty$ and $u$ be the Galerkin type weak solutions of the problems governed by the equations
	\begin{align}
	& u^n_{tt} + u^n_t -\Delta u^n = p^n_\varepsilon(t,u^n),\label{eq:u}\\
	& u_{tt} + u_t -\Delta u = p_\varepsilon(t,u),\label{eq:u2}
	\end{align}
	with the boundary data $u^n=u =0$ on $\partial \Omega$ and initial data $(u^n(0),u^n_t(0)) = (u^n_0,u^n_1)\in \mathcal{E}_0$ and $(u(0),u_t(0)) = (u_0,u_1)\in \mathcal{E}_0$. Choose $T>0$ such that $T> \sup_{n\in \mathbb{N}}\{t_n \}$. There hold the bounds
	\begin{align*}
	& \|\nabla u^n(t)\|_{L^2} \leq C, \ \ \|\nabla u(t)\|_{L^2} \leq C,\\
	& \|u^n_t(t)\|_{L^2} \leq C,  \ \ \|u_t(t)\|_{L^2} \leq C,\\
	& \|u^n_{tt}(t)\|_{H^{-1}} \leq C,  \ \ \|u_{tt}(t)\|_{H^{-1}} \leq C.
	\end{align*}
	for $t\in [0, T]$ with a constant $C>0$. 
	Moreover there hold the bounds
	$$
	\|u^n\|_{L^4(0,T;L^{12})} \leq C, \|u\|_{L^4(0,T;L^{12})} \leq C.
	$$
	This means that, for a subsequence
	\begin{align*}
	& u^n \to v \ \ \textrm{weakly-*}\ \ \textrm{in}\ \ L^\infty(0,T;H^1_0),\\
	& u^n_t \to v_t  \ \ \textrm{weakly-*}\ \ \textrm{in}\ \ L^\infty(0,T;L^2),\\
	& u^n_{tt} \to v_{tt} \ \ \textrm{weakly-*}\ \ \textrm{in}\ \ L^\infty(0,T;H^{-1}),	  
	\end{align*} 
	for a certain function $v \in L^\infty(0,T;H^1_0)$ with $v_t \in L^\infty(0,T;L^2)$ and $v_{tt} \in L^\infty(0,T;H^{-1})$. 
	By Lemma \ref{lem:cont} 
	$u^n, u\in C([0,T];\mathcal{E}_0)$.
	Moreover $v\in C([0,T];L^2) \cap  C_w([0,T];H^1_0)$, and, $v_t\in C([0,T];H^{-1})\cap C_w([0,T];L^2)$, cf \cite[Lemma 1.4, page 263]{Temam}. We will show that $v=u$ for $t\in [0,T]$.
	Note that for every $w\in L^2$ 
	$$
	(u^n(0),w) = (u^n(t),w) - \int_{0}^{t}(u^n_t(s),w)\, ds.
	$$
	Integrating with respect to $t$ between $0$ and $T$ and exchanging the order of integration we obtain 
	$$
	T(u^n_0,w) = \int_{0}^{T}(u^n(t),w)\, dt - \int_{0}^{T}(u^n_t(s),(T-s)w)\, ds.
	$$
	Passing to the limit we obtain 
	$$
	T(u_0,w) = \int_{0}^{T}(v(t),w)\, dt - \int_{0}^{T}(v_t(s),(T-s)w)\, ds = T(v(0),w),
	$$
	whence $v(0) = u_0$. It is straightforward to see that $u^n(t)\to v(t)$ weakly in $H^1_0$ for every $t\in [0,T]$. Similar reasoning for $u^n_t$ allows us to deduce that $v_t(0) = u_1$ and $u^n_t(t) \to v(t)$ weakly in $L^2$ for every $t\in [0,T]$.  
	Now we have to show that $v$ satsfies  \eqref{eq:u2}. Indeed, weak form  of \eqref{eq:u} is as follows 
	\begin{align*}
	& \int_{0}^{T}\langle u^{n}_{tt}(t), w(t)\rangle_{H^{-1}\times H^1_0{}}\, dt  + \int_{0}^{T}( u^{n}_{t}(t), w(t)) \, dt + \int_{0}^{T}(\nabla  u^{n}(t), \nabla w(t)) \, dt\\
	& \ \  = \int_{0}^{T}\int_{\Omega } f^n_{\epsilon}(u^n(x,t),t) w(t) \, dx \, dt,
	\end{align*}
	for every $w\in L^2(0,T;H^1_0)$. It suffices only to pass to the limit on the right-hand side. Fix $t\in [0,T]$ and $w\in H^1_0$. There holds $u^n(\cdot,t)\to u(\cdot,t)$ strongly in $L^{6-\frac{6}{5}\kappa}$ and, for a subsequence, $u^n(x,t)\to u(x,t)$ for a.e. $x\in\Omega$ and $|u^n(x,t)|\leq g(x)$ with $g\in L^{6-\frac{6}{5}\kappa}$ where $g$ can also depend on $t$. Hence 
	$$
	f^n_\epsilon(u^n(x,t),t)w(x) \to f_\epsilon(u(x,t),t)w(x)\ \ \textrm{a.e.}\ \ x\in \Omega,
	$$
	moreover 
	$$
	|f^n_\epsilon(u^n(x,t),t)w(x)| \leq C(1+|u^n(x,t)|^{5-\kappa}) |w(x)| \leq |w(x)|^6 + C(1+g(x)^{6-\frac{6}{5}\kappa})\in L^1.
	$$
	This means that
	$$
	\lim_{n\to \infty} \int_{\Omega } f^n_{\epsilon}(u^n(x,t),t) w(x) \, dx = \int_{\Omega } f_{\epsilon}(u(x,t),t) w(x) \, dx.
	$$ 
	Now let $w\in L^2(0,T;H^1_0)$. There holds 
	\begin{align*}	 
	&\left| \int_{\Omega } f^n_{\epsilon}(u^n(x,t),t) w(x,t) \, dx  \right| \leq  C  \int_{\Omega }(1+|u^n(x,t)|^{5}) |w(x,t)| \, dx  \leq C\|w(t)\|_{L^6}(1+\|u^n(t)\|_{L^6}^5)\\
	& \ \ \leq C\|w(t)\|_{H^1_0}(1+\|u^n(t)\|_{H^1_0}^5) \leq C\|w(t)\|_{H^1_0} \in L^1(0,T+1),
	\end{align*}
	whence we can pass to the limit in the nonlinear term. The fact that the $L^4(0,T;L^{12})$ estimate on $u^n$ is independent of $n$ implies that $v$ satisfies the same estimate which ends the proof that $u=v$. 
	
	We must show that $\norm {(u^n(t_n),u^n_t(t_n))-(u(t),u_t(t))}_{\mathcal{E}_0}\to 0$ 
	We already know that $u^n(t) \to u(t)$ weakly in $H^1_0$ and $u^n_t(t) \to u_t(t)$ weakly in $L^2$ for every $t\in [0,T]$. We will first prove that these convergences are strong. To this end let $w^n = u^n-u$. There holds
	$$
	w^n_{tt} + w^n_t - \Delta w^n = f^n_{\epsilon}(u^n,t) - f_\epsilon(u,t).
	$$ 
	Testing this equation with $w^n_t$ we obtain
	$$
	\frac{1}{2}\frac{d}{dt} \norm{(w^n(t),w^n_t(t))}_{\mathcal{E}_0}^2 
	+ \norm{w^n(t)}^2= \int_{\Omega}(f^n_{\epsilon}(t,u^n) - f_\epsilon(t,u))w^n_t(t)\, dx.  
	$$
	Simple computations lead us to 
	$$
	\frac{d}{dt}\norm{(w^n(t),w^n_t(t))}_{\mathcal{E}_0}^2 \leq \int_{\Omega}(f^n_{\epsilon}(t,u^n) - f^n_\epsilon(u,t))^2\, dx + \int_{\Omega}(f^n_{\epsilon}(t,u) - f_\epsilon(t,u))^2\, dx.  
	$$
	After integration from $0$ to $t$ we obtain 
	\begin{align*}
	&\norm{(w^n(t),w^n_t(t))}_{\mathcal{E}_0}^2 \leq 
	\norm{(u^n_0-u^n,u^n_1-u_1)}_{\mathcal{E}_0}^2 \\
	& \qquad + \int_{0}^T\int_{\Omega}(f^n_{\epsilon}(u^n,s) - f^n_\epsilon(u,s))^2\, dx\, ds + \int_{0}^T\int_{\Omega}(f^n_{\epsilon}(u,s) - f_\epsilon(u,s))^2\, dx\, ds.  
	\end{align*}
	We must pass to the limit in two terms. To deal with the first term observe that,
	\begin{align}
	& \int_{0}^T\int_{\Omega}(f^n_{\epsilon}(u^n,s) - f^n_\epsilon(u,s))^2\, dx\, ds \leq 
	\int_{0}^T\int_{\Omega}(C|u^n(s)-u(s)|(1+|u^n(s)|^{4-\kappa}+|u(s)|^{4-\kappa}))^2\, dx\, ds\nonumber\\
	& \ \ \leq C  \int_{0}^T\int_{\Omega}|u^n(s)-u(s)|^2(1+|u^n(s)|^{8-2\kappa}+|u(s)|^{8-2\kappa})\, dx\, ds \nonumber\\
	& \ \ \leq C \left(\|u^n-u\|^2_{L^2(0,T;L^2)} + \int_{0}^T\, \|u_n(s)-u(s)\|_{L^{\frac{12}{2+\kappa}}}^{2}\left(\|u^n(s)\|_{L^{12}}^{\frac{12}{8-2\kappa}}+\|u(s)\|_{L^{12}}^{\frac{12}{8-2\kappa}}\right)ds\right)\nonumber\\
	& \ \ \leq C \left(\|u^n-u\|^2_{L^2(0,T;L^2)} + \left(\int_{0}^T\, \|u_n(s)-u(s)\|_{L^{\frac{12}{2+\kappa}}}^{\frac{8-2\kappa}{1-\kappa}}\, dt\right)^{\frac{1-\kappa}{4-\kappa}} \left(\|u^n\|_{L^4(0,T;L^{12})}^{\frac{12}{4-\kappa}}+\|u\|_{L^4(0,T;L^{12})}^{\frac{12}{4-\kappa}}\right)\right)\nonumber\\
	& \ \ \leq C \left(\|u^n-u\|^2_{L^2(0,T;L^2)} + \|u^n-u\|_{L^{\frac{8-2\kappa}{1-\kappa}}(0,T;L^{\frac{12}{2+\kappa}})}^{\frac{8-2\kappa}{4-\kappa}} \right),\label{eq:conv_a}
	\end{align}
	and the assertion follows from the compact embedding $H^1_0\subset L^{\frac{12}{2+\kappa}}$ by the Aubin--Lions lemma. To deal with the second term note that $f^n_\epsilon(u,t)\to f_\epsilon(u,t)$ for almost every $(x,t)\in \Omega\times (0,T)$. Moreover 
	
	$$
	(f^n_{\epsilon}(u,t) - f_\epsilon(u,t))^2 \leq C,
	$$
	and the Lebesgue dominated convergence theorem implies the assertion. 
	
	Now, the triangle inequality implies  
	\begin{align*}
	& \|\nabla u^n(t_n)-\nabla u(t)\|_{L^2}^2 + \|u^n_t(t_n)- u_t(t)\|_{L^2}^2 \\
	& \ \ \leq 2\left(	\|\nabla u^n(t_n)-\nabla u(t_n)\|_{L^2}^2 + \|u^n_t(t_n)- u_t(t_n)\|_{L^2}^2\right) \\
	& \qquad + 	2\left(\|\nabla u(t_n)-\nabla u(t)\|_{L^2}^2 + \|u_t(t_n)- u_t(t)\|_{L^2}^2\right),
	\end{align*}
	where both terms tend to zero, the first one by \eqref{eq:conv_a} and the second one by Lemma \ref{lem:cont} and the proof is complete. 
\end{proof}
\section{Existence and regularity of non-autonomous attractors.}\label{sec:6}

We start from the result which states that the solution can be split into the sum of two functions: one that decays to zero, and another one which is more smooth than the initial data.
\begin{lemma}\label{le:spliting}
Let $u$ be the Shatah--Struve solution of \eqref{eq:prblm} such that $u(t) \in B_0$ for every $t\geq 0$, where $B_0$ is the absorbing set from Proposition \ref{prop:abball}.
There exists the increasing sequence $\alpha_0,\ldots,\alpha_k$ with $\alpha_0=0$, $\alpha_k=1$ such that
if $\norm {(u(t),u_t(t))}_{\mathcal{E}_{\alpha_i}}\leq R$ for every $t\in[0,\infty)$, then $u$ can be represented as the sum of two functions $v,w$ satisfying
\begin{align*}
& u(t)=v(t)+w(t), \ \ \norm{(v(t),v_t(t))}_{\mathcal{E}_{\alpha_{i}}}\leq  \norm{(u_0,u_1)}_{\mathcal{E}_{\alpha_i}} C e^{-t\alpha} \\
& \qquad \textrm{and}\ \  \norm{(w(t),w_t(t))}_{\mathcal{E}_{\alpha_{i+1}}} \leq C_R \ \ \textrm{for}\ \  i\in \{0,\ldots,k-1\}.
\end{align*}

Moreover constants $C,C_R ,\alpha$ are the same for every $p(t,u)\in \mathcal{H}_{[0,1]}$ treated as the right-hand side  in equation \eqref{eq:prblm}.
\end{lemma}
\begin{proof}

From the Gagliardo-–Nirenberg interpolation inequality we see that
$$\norm{p_\varepsilon(t,u)}_{H^{\alpha}}\leq
C \norm{\nabla p_\varepsilon(t,u)}_{L^s}^\theta \norm{p_\varepsilon(t,u)}_{L^q}^{1-\theta}+C\norm{p_\varepsilon(t,u)}_{L^q}$$
with $\alpha\leq\theta\leq1$,
$\frac{1}{2}=\frac{\alpha}{3}+\left(\frac{1}{s}-\frac{1}{3}\right)\theta+\frac{1-\theta}{q}$ and $s<2$.
From the H{\"o}lder inequality we obtain
$$
\norm{p_\varepsilon(t,u)}_{H^\alpha}\leq
C
\left (\int_{\Omega}\left|\partu{p_\varepsilon}(t,u)\right|^{sp_*}dx \right)^\frac{\theta}{sp_*}
\left(\int_{\Omega}|\nabla u|^{sp}dx\right)^\frac{\theta}{sp}
\norm{p_\varepsilon(t,u)}_{L^q}^{1-\theta}+C\norm{p_\varepsilon(t,u)}_{L^q}.
$$
From assumption \hyperref[as:H5]{(H5)}, the Cauchy inequality, and the fact that solution $u$ is included in the absorbing set, taking $sp = 2,\;sp^{*}=3,\;\theta=\frac{1}{2}$ we get the inequality
\begin{align}
&\norm{p_\varepsilon(t,u)}_{H^\alpha}\leq
C(R) \left(
\left(\int_{\Omega}|u|^{12}dx\right)
^\frac{1}{3}
+
\left(\int_{\Omega}|u|^{(5-\kappa)q}dx\right)^{\frac{1}{q}}
+1
\right)\;\nonumber\\
&\qquad \qquad \qquad \qquad \qquad
\text{with}\; \alpha = \frac{3}{2}\left(\frac{1}{2}-\frac{1}{q}\right),\; \alpha<\frac{1}{2}
\label{in:fracnorm}
\end{align}
Now we will inductively describe sequence $\alpha_1\ldots,\alpha_{k-1}$ starting with $\alpha_1$.
If we set $ \frac{5-\kappa}{10}\leq \frac{1}{q} < \frac{1}{2}$ in inequality \eqref{in:fracnorm}, we obtain
$$
\int_{t_0}^{t_0+h}
\norm{p_\varepsilon(t,u)}_{H^{\alpha_1}}dt\leq
C(R)\left(\norm{u}_{L^4(t_0,t_0+h;L^{12})}^4
+\norm{u}_{L^5(t_0,t_0+h;L^{10})}^5
+h\right)
\leq
C(h,R).
$$
We  observe that $\alpha_1\in (0,\delta )$, for some $\delta>0$.
Assume that $i\in\{1,\dots,k-1\}$
$$
\norm{(u(t),u_t(t))}_{\mathcal{E}_{\alpha_i}}\leq R \quad \text{and}\quad
\int_{t_0}^{t_0+h} \norm{p_\varepsilon(t,u)}_{H^{\alpha_i}}dt\leq C(h,R)\ \quad\text{for $t,t_0\in[0,\infty)$} .
$$
From Lemma \ref{le:fractionlstrichartz}  we see that
$$
 u\in L^4(t_0,t_0+h;W^{\alpha_i,12}),
 \qquad
\norm{u}_{L^4(t_0,t_0+h;W^{\alpha_i,12})}\leq C(h,R) .
$$
By the Sobolev embeding $W^{\alpha_i,10}\hookrightarrow L^{\frac{30}{3-10\alpha_i}} $ and by  interpolation we see that
$$
\norm{u}_{L^5\left(t_0,t_0+h;L^{\frac{30}{3-10\alpha_i}} \right)}
\leq
\norm{u}_{L^5(t_0,t_0+h;W^{\alpha_i,10})}
\leq
\norm{u}_{L^4(t_0,t_0+h;W^{\alpha_i,12})}^\frac{4}{5}\norm{u}_{L^\infty(t_0,t_0+h;H^{\alpha_i +1})}^\frac{1}{5} \leq C(h,R).
$$
Using \eqref{in:fracnorm} with $q=\frac{6}{3-10\alpha_i}$ we  obtain 
$$
\int_{t_0}^{t_0+h}
\norm{p_\varepsilon(t,u)}_{H^{\alpha_{i+1}}}dt\leq
C(R)\left(\norm{u}_{L^4(t_0,t_0+h;L^{12})}^4
+\norm{u}_{L^5\left(t_0,t_0+h;L^{\frac{30}{3-10\alpha_i}}\right)}^5
+h\right)
\leq
C(h,R),
$$
with $\alpha_{i+1}=\frac{5}{2}\alpha_i$. From this recurrent relation and the fact that $\alpha_1\in(0,\delta)$ we can find sequence $\alpha_1,\ldots,\alpha_{k-1}$ such that $\alpha_{k-1} = \frac{9}{20}$. Let observe that in case $\alpha_{k-1} = \frac{9}{20}$ from the Sobolev embedding we get the bounds
$$
\norm{u}_{L^{60}}\leq C(R) \quad\text{and}\quad \norm{\nabla u}_{L^{\frac{60}{21}}}\leq C(R).
$$
Hence,
$$
\norm{\nabla p_\varepsilon(t,u)}_{L^2} \leq C\left(1+ \int_\Omega |u|^8 |\nabla u|^2\, dx\right)
\leq
C\left(1+\norm{u}_{L^{28}}^8 \norm{\nabla u}_{L^{\frac{60}{21}}}^2\right)\leq C(R),
$$
and consequently there holds $$\int_{t_0}^{t_0+h}\norm{\nabla p_\varepsilon(t,u)}_{L^2}\,  dt \leq C(h,R).$$
\\Let us decompose $u(t)=w(t)+v(t)$ where $w,v$ satisfy the problems
$$
\begin{cases}
v_{tt}+v_t-\Delta v = 0,
\\v(t,x) = 0\; \text{for}\; x\in \partial \Omega,
\\v(0,x) = u_0(x),
\\v_t(0,x) = u_1(x),
\end{cases}
\qquad\qquad
\begin{cases}
w_{tt}+w_t-\Delta w = p_\varepsilon(t,v+w),
\\w(t,x) = 0\; \text{for}\; x\in \partial \Omega,
\\w(0,x) = 0,
\\w_t(0,x) = 0.
\end{cases}
$$
From Lemma  \ref{le:fractionlstrichartz}  we get that 
$\norm{(v(t),v_t(t))}_{\mathcal{E}_{\alpha_i}} \leq C \norm{(u_0,u_1)}_{\mathcal{E}_{\alpha_i}} e^{- \alpha t}$ 
and
$$\norm{(w(t+h),w_t(t+h))}_{\mathcal{E}_{\alpha_{i+1}}}\leq
C e^{-\beta h}\norm{(w(t),w_t(t))}_{\mathcal{E}_{\alpha_{i+1}}}+ C(h,R),
$$
for every $t\geq 0$ and $h>0$. We set $h$ such that $ Ce^{-\beta h}\leq \frac{1}{2}$. Then we obtain that $\norm{(w(t),w_t(t))}_{\mathcal{E}_{\alpha_{i+1}}}\leq 2C(h,R) = C_R$ for $i\in\{0,\ldots,k-1 \}$.
We stress that all constants are independent of $p_\varepsilon(t,u)\in\mathcal{H}_{[0,1]}$.
\end{proof}
Bounds obtained in the previous lemma allow us to deduce the asymptotic compactness of the considered non-autonomous dynamical system.

\begin{proposition}
For every $\varepsilon\in[0,1]$, the non-autonomous dynamical system $(\varphi_\varepsilon,\theta)$ is uniformly asymptotically compact.
\end{proposition}
\begin{proof}
Let $B_0$ be an absorbing set from Proposition \ref{prop:abball}. Then for every bounded set $B\subset \mathcal{E}$ there exist $t_0$ such that for every $t\geq t_0$ and every $p_\varepsilon\in \mathcal{H}(f_\varepsilon)$ there holds  $\varphi_\varepsilon(t,p_\varepsilon)\in B_0$. From the previous lemma
there exists the set $B_{\alpha_{1}}\subset \mathcal{E}_{\alpha_1}$ which is compact in $\mathcal{E}_0$ such that
$$
\lim_{t\to \infty}
\sup_{p_\varepsilon\in \mathcal{H}(f_\varepsilon)} 
\dist(\varphi(t,p)B,B_{\alpha_1}) =0,
$$
which shows that the non-autonomous dynamical system $(\varphi_\epsilon,\theta)$ is uniformly asymptotically compact.
\end{proof}
We are in position to formulate the main result of this section, the theorem on non-autonomous attractors. 
\begin{theorem}\label{thm:main}
For every $\varepsilon\in[0,1]$ problem \eqref{eq:prblm} has uniform $\mathcal{A}_\varepsilon$, cocycle $\{\mathcal{A}_\varepsilon(p)\}_
{p\in\mathcal{H}(f_\epsilon)}$ and pullback attractors which are bounded in $\mathcal{E}_1$ uniformly with respect to $\varepsilon$. Moreover there holds
$$
\mathcal{A}_\varepsilon=
\bigcup_{p\in\mathcal{H}(f_\epsilon)}
\mathcal{A}_\varepsilon(p)
$$

\begin{proof}
Because $(\varphi_\epsilon,\theta)$ is asymptotically compact, from Theorem \ref{NoAtr-Rel} we get existence of uniform and cocycle attractors and the relation between them. For $(u_0,u_1) \in \mathcal{A}_\varepsilon$ by Theorem \ref{lifted-inv} there exists the global solution $u(t)$ with $(u(0),u_t(0)) = (u_0,u_1)$.
If $\mathcal{A}_\varepsilon$ is bounded in $\mathcal{E}_{\alpha_i}$ then from Lemma \ref{le:spliting}  we can split this solution into the sum $u(t)=v^n(t)+w^n(t)$ for $t\in [-n,\infty)$ such that
$$
\text
{
$\norm{(v^n(t),v^n_t(t))}_{\mathcal{E}_{\alpha_{i}}}\leq C e^{-(t+n)\alpha}$ 
and 
$\norm{(w^n(t),w^n_t(t))}_{\mathcal{E}_{\alpha_{i+1}}} \leq C$.
}
$$
Then, for the subsequence, there holds $w^n(0) \to w$ and $v^n(0) \to 0$ as $n\to \infty$ for some $w\in \mathcal{E}_{\alpha_{i+1}}$, so $w = (u_0,u_1)$. Because $\mathcal{A_\varepsilon}$ is bounded in $\mathcal{E}_0$  in finite number of steps we obtain the boudedness of the uniform attractors in $\mathcal{E}_1$. Moreover, due to  Proposition \ref{prop:abball} and Lemma \ref{le:spliting} the $\mathcal{E}_1$ bound of these attractors does not depend on $\varepsilon$.

\end{proof}

\end{theorem}

\section{Upper semicontinuous convergence of attractors.}\label{sec:7}
The paper is concluded with the result on upper-semicontinuous convergence of attractors.
\begin{theorem}
 The family of uniform attractors $\{\mathcal{A}_\varepsilon\}_{\varepsilon\in [0,1]}$ for the considered non-autonomous dynamical system $(\varepsilon,\theta_t)$ is upper semi-continuous in Kuratowski and Hausdorff sense in $\mathcal{E}_0$ as $\varepsilon \to 0$.
\end{theorem}
\begin{proof}
Let $(u_0^n,u_1^n)\in \mathcal{A}_{\varepsilon_n} $ such that
$(u_0^n,u_1^n)\to (u_0,u_1)$ in $\mathcal{E}_0$. There exist $p_{\varepsilon_n} \in\mathcal{H}_{[0,1]}$ such that there exist global solution $u_n(t,x)$ to problem
$$
\begin{cases}
u^n_{tt}+u^n_t-\Delta u^n = p_{\varepsilon_n}(t,u),
\\u^n(t,x) = 0\; \text{for}\; x\in \partial \Omega,
\\u^n(0,x) = u^n_0(x),
\\u^n_t(0,x) = u^n_1(x).
\end{cases}
$$
As in the proof of Proposition \ref{eq:propcont} it follows that for every $T$ there exist $v\in L^\infty(T,-T;H^1_0)$ with 
$v_t\in L^\infty(T,-T;H^1_0), v_{tt}\in  L^\infty(T,-T;H^{-1})$ and $v\in L^4(-T,T;L^{12})$   such that for the subsequence of $u^n$ there hold the convergences
\begin{align*}
	 & u^n \to v \ \ \textrm{weakly-*}\ \ \textrm{in}\ \ L^\infty(-T,T;H^1_0),\\
	 & u^n_t \to v_t  \ \ \textrm{weakly-*}\ \ \textrm{in}\ \ L^\infty(-T,T;L^2),\\
	  & u^n_{tt} \to v_{tt} \ \ \textrm{weakly-*}\ \ \textrm{in}\ \ L^\infty(-T,T;H^{-1}).	
\end{align*} 
Moreover $(u^n(t),u^n_t(t))\to (v(t),v_t(t))$ weakly in $\mathcal{E}_0$ for every $t\in[-T,T]$ which implies that $(v(0),v_t(0)) = (u_0,u_1)$ and $u^n(t)\to v(t)$ strongly in $L^2$. We will show that $v$ is a weak solution for the autonomous problem, i.e., the problem with $\varepsilon = 0$. It is enough to show that for every 
$w\in L^2(-T,T;H^1_0)$ there holds
$$
\lim_{n\to \infty}\int_{-T}^T (p_{\varepsilon_n}(u^n(t) - f_0(v^n(t)),w(t)) \, dt
=0
$$ 
Let observe that $\norm{u_n(t)}_{C^0}\leq R$ and $\norm{v(t)}_{C^0}\leq R$ due to the fact that all attractors are bounded uniformly in $\mathcal{E}_1$ and the Sobolev embedding $H^2\hookrightarrow C^0$. Hence
\begin{align*}
  & \left| \int_{-T}^T (p_{\varepsilon_n}(u^n(t),t)-f_0(v(t)),w(t))dt \right|\\
& \qquad \leq
\int_{-T}^T|( p_{\varepsilon_n}(u^n(t),t) -f_0(u^n(t)),w(t))|dt +
\int_{-T}^T|( f_0(u^n(t)) -f_0(v(t)),w(t))|dt
\\ 
&\qquad \leq 
\sup_{t\in \mathbb{R}}\sup_{|s|\leq R} 
|( p_{\varepsilon_n}(s,t) -f_0(s))| \norm {w}_{L^1(-T,T;L^2)}
\\
& \qquad \qquad +\sup_{|s|\leq R} |f'_0(s)|
\left(\int_{-T}^T \norm{u^n(t)-v(t)}^2 dt\right)^{\frac{1}{2}}
\norm {w}_{L^2(-T,T;L^2)}^{\frac{1}{2}}
\end{align*}
Due to \hyperref[as:H2]{(H2)} the first term tends to zero. The second term also tends to zero by the Aubin--Lions lemma.  Hence, $v(t)$ is the weak solution on the  interval $[-T,T]$ with $v(0)=(u_0,u_1)$. By the diagonal argument we can extend $v$ to a global weak solution. Moreover, as $v$ is also the Shatah--Struve solution, it is  unique. Moreover $\norm{v(t)}_{\mathcal{E}_1}\leq C$ due to the uniform boudedness of attractors $\mathcal{A}_\varepsilon$ in $\mathcal{E}_1$. Hence $\{v(t)\}_{t\in \R}$ is a global bounded orbit for the autonomous dynamical system $\varphi_0$ which implies that $(u_0,v_0)\in\mathcal{A}_0 $ and shows the upper semi-continuity in the Kuratowski sense. Because all uniform attractors $\mathcal{A}_\varepsilon$ are uniformly bounded in $\mathcal{E}_1$, their sum $\cup _{\varepsilon\in [0,1]}\mathcal{A}_\varepsilon$ is relatively compact in $\mathcal{E}_0$. So, by Lemma \ref{th:semicontinuity} we have also upper semi-continuity in Hausdorff sense.

\end{proof}

\section{Appendix.}\label{sec:8}
\subsection{Non-autonomous attractors}
The results of this section can be found in \cites{Langa,Kloeden}.
\begin{definition}
  Let $X,\Sigma$ be metric spaces, $\{\theta_t\}_{t \geq 0}$ be a semigroup and $\varphi:\mathbb{R}^+\times \Sigma \to C(X)$ is family of continuous maps on $X$. Let the following conditions hold
  \begin{itemize}
    \item $\varphi(0,\sigma) = \textrm{Id}_X$ for every $\sigma\in\Sigma$.
    \item Map $ \mathbb{R}^+\times \Sigma \ni(t,\sigma)\to \varphi(t,\sigma)x \in X $ is continuous for every $x$.
    \item For every $t,s\geq 0$ and $\sigma \in \Sigma $ there holds the cocycle property
  $$\varphi(t+s,\sigma)=\varphi(t,\theta_s \sigma)\varphi(s, \sigma ) .$$

  \end{itemize}
  Then the pair $(\theta,\varphi)_{(X,\Sigma)}$ is called a non-autonomous dynamical (NDS) and map $\phi$ a cocycle semiflow.
\end{definition}

\begin{definition}
  The set $\mathcal{A}\subset X$ is called uniform attractor for the cocycle $\varphi$ on $\Sigma,X$ if
  $\mathcal{A}$ is smallest compact set such that for every bounded sets $B\subset X$ and $\Upsilon	\subset \Sigma$ for every holds
  $$
  \lim_{t\to \infty}\sup_{\sigma \in \Upsilon }  \dist(\varphi(t,\sigma)B,\mathcal{A})\to 0.
  $$
\end{definition}

\begin{definition}
  Let $(\theta,\varphi)_{(X,\Sigma)}$ be an NDS such that $\theta$ is a group i.e $\Sigma$ is invariant for every  $\theta_t$. Then we call the family of compact sets $\{\mathcal{A}(\sigma) \}_{\sigma\in\Sigma}\subset \mathcal{P}(X)$ a cocycle atractor if  there holds
  \begin{itemize}
    \item $\{\mathcal{A}(\sigma)\}_{\sigma\in \Sigma}$ is invariant under the NDS $(\theta,\varphi)_{(X,\Sigma)}$, i.e.,
    $$
    \varphi(t,\sigma)\mathcal{A}(\sigma) = \mathcal{A}(\theta_t \sigma)
    \;\text{ for every $t\geq0$.}
    $$
    \item  $\{\mathcal{A}(\sigma)\}_{\sigma\in \Sigma}$ pullback attracts all bounded subsets $B\subset X$, i.e.,
    $$
    \lim_{t\to\infty}\dist(\varphi(t,\theta_{-t} \sigma) B,\mathcal{A}(\sigma) )) = 0.
    $$
  \end{itemize}
\end{definition}
\begin{remark}
If for some $\sigma\in\Sigma$ we  consider the mapping $S(t,\tau) =\varphi(t,\sigma)$ for an NDS $(\theta,\varphi)_{(X,\Sigma)}$ then the family of mappings $\{S(t,\tau):t\geq \tau \}$ forms an evolution process. Let $\{\mathcal{A}(\sigma) \}_{\sigma\in\Sigma}$ be a cocycyle atrator for NDS.
Then $\mathcal{A}(t) = \mathcal{A}(\theta_t \sigma)$ is called a pullback atractor for $S(t,\tau)$.
\end{remark}
\begin{definition}
  We say that the NDS $(\theta,\varphi)_{(X,\Sigma)}$ is uniformly asymptotically compact if there exist the compact set $K\subset X$ such that
  $$
    \lim_{t\to \infty}\sup_{\sigma \in \Upsilon }
    \dist(\varphi(t,\sigma)B,K)\to 0 .
  $$
\end{definition}

\begin{theorem}\cite[Theorem 3.12.]{Langa}\label{NoAtr-Rel}
  Suppose that the NDS $(\theta,\varphi)_{(X,\Sigma)}$ is
  uniformly asymptotically compact with $\Sigma$ which is compact and invariant under the flow $\theta$. Then the uniform and cocycle attractors exist and there holds
  $$\bigcup_{\sigma\in\Sigma}\mathcal{A}(\sigma) = \mathcal{A}, $$
  where
  $ \{\mathcal{A}(\sigma)\}_{\sigma \in \Sigma}$  is cocycle atractor and
  $\mathcal{A}$ is uniform atractor.
\end{theorem}
\begin{definition}
    Let $(\theta,\varphi)_{(X,\Sigma)}$ be an NDS such that $\theta$ is a group. We call $\xi:\mathbb{R}\to X$ a global solution through $x$ and $\sigma$ if, for all $t\geq s$ it satisfies
    $$
    \varphi(t-s,\theta_s \sigma)\xi(s) = \xi(t)
    \text{ and }
    \xi(0) = x.
    $$
\end{definition}
\begin{definition}\label{lifted}
    Let $(\theta,\varphi)_{(X,\Sigma)}$ be an NDS. We say that a subset $\mathcal{M}\subset X$ is lifted-invariant if for each $x\in \mathcal{M}$ there exist $\sigma$ and bounded global solution $\xi:\mathbb{R}\to X$ through $x$ and $\sigma$.
\end{definition}
\begin{theorem}\cite[Proposition 3.21]{Langa}\label{lifted-inv}
  Suppose that the NDS $(\theta,\varphi)_{(X,\Sigma)}$ is
  uniformly asymptotically compact with $\Sigma$ which is compact and invariant under flow $\theta$. Then the uniform attractor $\mathcal{A}$ is the maximal bounded lifted-invariant set of the NDS $(\theta,\varphi)_{(X,\Sigma)}$.
\end{theorem}
\subsection{Upper semicontinuity} We recall the definitions of Hausdorff and Kuratowski upper-semicontinuous convergence of sets, and the relation between these conditions.
\begin{definition}
Let $(X,d)$ be a metric space and let $\{A_\varepsilon\}_{\varepsilon\in[0,1]}$ be a family of sets in $X$. We say that this family converges to $A_0$  upper-semicontinuously in Hausdorff sense if
$$
\lim_{\varepsilon\to 0^+}\text{dist}_X(A_\varepsilon,A_0) = 0.
$$

\end{definition}
\begin{definition}
Let $(X,d)$ be a metric space and let $\{A_\varepsilon\}_{\varepsilon\in[0,1]}$ be a family of sets in $X$. We say that this family converges to $A_0$  upper-semicontinuously in Kuratowski sense if
$$
X-\limsup_{\varepsilon\to 0^+}A_\varepsilon \subset A_0,
$$
where $X-\limsup_{\varepsilon\to 0^+}A_\varepsilon$ is the Kuratowski upper limit defined by 
$$
X-\limsup_{\varepsilon\to 0^+}A_\varepsilon
=\{x\in X:\lim_{\varepsilon_n\to 0}d(x_{\varepsilon_n},x)=0,\;x_{\varepsilon_n}\in A_{\varepsilon_n} \}.
$$
\end{definition}
The proof of the next result can be found for example in \cite[Proposition 4.7.16]{DMP}.
\begin{lemma}\label{th:semicontinuity} 
Assume that the sets $\{A_\varepsilon \}_{\varepsilon\in[0,1]}$ are nonempty and closed and the set $\cup_{\varepsilon\in[0,1]}A_{\varepsilon}$ is relatively compact. Then if the family $\{A_\varepsilon\}_{\varepsilon\in[0,1]}$ converges to $A_0$ upper-semicontinuously in Kuratowski sense then $\{A_\varepsilon\}_{\varepsilon\in[0,1]}$ converges to $A_0$ upper-semicontinuously in Hausdorff sense.
\end{lemma}

\end{document}